\documentclass[12pt]{article}
\usepackage[utf8]{inputenc}
\usepackage[english]{babel}
\usepackage{amsthm}
\usepackage{amsmath, amssymb, mathrsfs}

\usepackage[colorlinks=true,linktocpage=true,linkcolor=blue,citecolor=red]{hyperref}
\usepackage[capitalise]{cleveref}

\usepackage{tikz,tikz-cd}
\usepackage{xypic}
\usepackage{graphicx}
\usepackage{amsfonts}

\usepackage{xcolor}

\usepackage{aliascnt}
\newtheorem{theorem}{Theorem}[section]
\newtheorem{corollary}[theorem]{Corollary}
\newtheorem{proposition}[theorem]{Proposition}
\newtheorem{lemma}[theorem]{Lemma}
\theoremstyle{definition}
\newtheorem{definition}[theorem]{Definition}

\theoremstyle{definition}
\newtheorem*{definition*}{Notational Conventions}

\theoremstyle{remark}
\newtheorem{remark}[theorem]{Remark}

\theoremstyle{definition}
\newtheorem{example}[theorem]{Example}

\theoremstyle{definition}
\newtheorem{construction}[theorem]{Construction}

\theoremstyle{definition}
\newtheorem{assumption}[theorem]{Assumption}

\newcommand{\uvar}{\mathord{\relbar}}

\newcommand{\icat}{\mathbf{Cat}_\infty}

\newcommand{\Sp}{\mathcal{S}}

\newcommand{\Mon}{\mathbf{Mnd}}

\newcommand{\LMod}{\mathbf{LMod}}

\newcommand{\Monoid}{\mathbf{Mon}}

\newcommand{\Md}{\mathbf{RMd}}

\newcommand{\Ladj}{\mathbf{LAdj}}

\newcommand{\PTh}{\operatorname{\mathbf{PreTh}}}

\newcommand{\Th}{\operatorname{\mathrm{Th}}}
\newcommand{\Mt}{\mu}

\newcommand{\Kl}{\mathrm{Kl}}

\newcommand{\colim}{\operatorname{\mathrm{Colim}}}

\newcommand{\sSet}{\text{Set}_\Delta}

\newcommand{\Map}{\mathrm{Map}} 
\newcommand{\Fun}{\mathrm{Fun}} 

\newcommand{\End}{\operatorname{\mathrm{End}}}

\newcommand{\EndO}{\operatorname{\underline{\mathrm{End}}}}

\newcommand{\Prsh}{\operatorname{\mathrm{Pr}}}

\newcommand{\Mod}{\operatorname{\mathrm{Mod}}}

\title{Higher Theories and Monads}
\author{Simon Henry and Nicholas J. Meadows}

\begin{document}

\maketitle

\begin{abstract}
We extend Bourke and Garner's idempotent adjunction between monads and pretheories to the framework of $\infty$-categories and we use this to prove many classical results about monads in the $\infty$-categorical framework. Amongst other things, we show that the category of algebras for an accessible monads on a locally presentable $\infty$-category $\mathcal{E}$ is again locally presentable, and that a diagram of accessible monads on a locally presentable $\infty$-category admits a colimit. Our results also provide a new and simpler way to construct and describe monads in terms of theories. 
\end{abstract}

\tableofcontents

\section{Introduction}

At the present time, monads on $\infty$-categories are arguably difficult to work with. In \cite{lurieHA}, Jacob Lurie developed a relatively nice theory of monads on $\infty$-categories as a byproduct of his theory of $\infty$-operads and proved the Barr-Beck monadicity theorem for $\infty$-categories. Essentially, a monad is defined there as a monoid object in the monoidal $\infty$-category of endofunctors. However, this theory remains relatively difficult to use in practice due to the fact that unpacking all the definitions involved in the previous sentence takes a lot work (we review this in \cref{sec:monads_prelim}). Also many classical theorems about monads have not yet been proven in this context. For example, it does not seem possible to deduce from \cite{lurieHA}\footnote{Lurie's work contains some results about colimits in category of algebras, but as far as we know, in the case of monads they only applies when the monad preserves colimits and hence colimits of algebras are just colimits in the underlying category.} that the category of algebras for an accessible monad on a cocomplete category has all colimits.

Riehl and Verity proposed an alternative, simpler, definition of monads in \cite{riehl2016homotopy} for which they also proved the Barr-Beck monadicity criterion. But it is also more model dependent than Lurie's definition as it relies on a strict action of a simplicial monoid on a quasi-category.

This paper is meant to be a toolbox filling some of these gaps and offering a new way to work with (most) monads on $\infty$-categories using only basic $\infty$-category theory instead of Lurie's theory of operads and in an essentially model independent way. This is mostly based on an $\infty$-categorical adaptation of the work on Bourke and Garner in \cite{bourke2019monads} for $1$-categorical monads.

Versions of the monad-theory adjunction have appeared in the category theory literature since the 1960s, beginning with Linton's result (\cite{linton1966some}). In \cite{bourke2019monads}, Bourke and Garner developed a very general monad-theory adjunction, which encompassed many, if not all, of the previously known constructions. Disregarding the enriched category theoretic aspect for simplicity, if $\mathcal{A} \subset \mathcal{E}$ is a small dense full subcategory, an $\mathcal{A}$-pretheory is just a bijective on objects (or essentially surjective) functors $\mathcal{A} \to \mathcal{K}$, with $\mathcal{K}$ a small $\infty$-category. Any monad $M$ on $\mathcal{E}$ has an attached pretheory, called its theory, which is the full subcategory of the Kleisli category of $M$ of objects that are in $\mathcal{A}$.

Given an $\mathcal{A}$-pretheory $\mathcal{A} \to \mathcal{K}$ one defines the category of $\mathcal{K}$-models in $\mathcal{E}$ as objects $X \in \mathcal{E}$ whose restricted Yoneda embeddings in $\Prsh(\mathcal{A})$ have an extension to a presheaf on $\mathcal{K}$. That is, it can be expressed as as a pullback square:

\[\begin{tikzcd}
  \Mod_{\mathcal{E}}(\mathcal{K}) \ar[d] \ar[r] \ar[dr,phantom,"\lrcorner"very near start] & \Prsh(\mathcal{K}) \ar[d] \\
\mathcal{E} \ar[r] & \Prsh(\mathcal{A}) 
\end{tikzcd} \]

Now, Bourke and Garner show that under the assumption that $\mathcal{E}$ is locally presentable, the functor $\Mod_{\mathcal{E}}(\mathcal{K}) \to \mathcal{E}$ is a monadic right adjoint. In particular, it gives a monad $\Mt^\mathcal{K}$ associated to $\mathcal{K}$ which is characterized by the property that $\Mt^\mathcal{K}$-algebras are the same as $\mathcal{K}$-models.

Finally, they show that these two constructions (from monads to pretheory and pretheory to monads) are adjoint to each other and form an idempotent adjunction, i.e. induces an equivalence of categories between their essential images. The object in the images are respectively called $\mathcal{A}$-theories, and $\mathcal{A}$-Nervous monads, as they are exactly the monads that satisfy the conclusion of the nerve theorem.

In the present paper, we will generalize these results to the $\infty$-categorical context. While Bourke and Garner generalize all this to an enriched setting (where $\mathcal{E}$, $\mathcal{A}$ and $\mathcal{K}$ are all $V$-enriched categories and $M$ is a $V$-enriched monad for $V$ a nice enough monoidal category), we will restrict to the unenriched setting (as presented above) as we feel the theory of enriched $\infty$-categories is not yet developed enough for this.

\bigskip

In \cref{sec:Monads_as_Kleisli} we also show that the category of monads on an $\infty$-category $\mathcal{C}$ is equivalent (though the construction of the Kleisli category) with the $\infty$-category of essentially surjective left adjoint functors $\mathcal{C} \to \mathcal{K}$. This result is not directly related to the main goals of the paper, but it follows from the methods developed in the paper and is fairly similar to the construction of the Monad-theory adjunction. This result produce a much simpler description of the $\infty$-category of monads, which is why we decided to include it.

\bigskip

The main kind of application of our results is to deduce several structural theorems about monads, such as the existence of colimits of monads and colimits in the $\infty$-category of algebras for a monad, by looking instead at colimits of theories and colimits in the category of models of a theory. In order to do this, one needs to show that most monads are actually $\mathcal{A}$-nervous monads for $\mathcal{A}$ a large enough dense subcategory. This is achieved using an $\infty$-categorical generalization of the work of Berger, Mellies and Weber in \cite{berger2012monads} where they showed that a large class of monads, that they call ``monads with arities'', satisfy a nerve theorem (that is are nervous monads). In particular, their results show that any $\lambda$-accessible monad on a locally $\lambda$-presentable category is $\mathcal{A}$-nervous for $\mathcal{A}$ the full subcategory of $\lambda$-presentable objects. We generalize this to accessible monads on $\infty$-categories in \cref{sec:general_consequence}. Using this, we show that:

\begin{itemize}

\item For any accessible monad on a locally presentable $\infty$-category the category of $M$-algebras is locally presentable, in particular it has all colimits. Indeed, the category of models of an $\mathcal{A}$-pretheory is easily seen to be locally presentable. See \cref{cor5.5}.

\item Any small diagram $I \to \Mon_\mathcal{E}$ of accessible monads on a locally presentable $\infty$-category $\mathcal{E}$ has a colimits in the $\infty$-category $\Mon_{\mathcal{E}}$ of monads on $\mathcal{E}$. Moreover an algebra for the colimit monad $\colim_i M_i$ is an object of $\mathcal{E}$ equiped with compatible structure of $M_i$ algebra for all $i$. More concretely, we have:

  \[ \mathcal{E}^{\colim_{i \in I} M_i} \simeq \lim_{i\in I} \mathcal{E}^{M_i}, \]

where $\mathcal{E}^M$ denotes the category of $M$-algebras for a monad $M$ and the limit on the right uses the forgetful functors induced by the morphisms of monads between the $M_i$. This is proven using the fact that colimits of $\mathcal{A}$-pretheories are easy to understand (they are just colimits in the $\infty$-category $\icat$ of $\infty$-categories) and the monad-theory adjunction preserves colimits. See \cref{cor:colimits_of_monads}.
\end{itemize}

\bigskip

A second type of application of our result is to construct examples of monads on $\infty$-categories from (pre)theories. Pretheories are much easier to work with directly, since they are just essentially surjective functors of $\infty$-categories. We treat in detail the case of the monads for $E_1,E_2$ and $E_\infty$ algebras in \cref{sec:E1_algebra}, and many other more involved examples are in \cref{sec:chu_haugseng}. In many of these examples $\mathcal{A}$ and $\mathcal{K}$ can be taken to be (nerve of) $1$-categories.

This can be thought more generally as a procedure to extend a classical monad $M_0$ on a $1$-category to an ``$\infty$-monad'' $M$ on an $\infty$-category by viewing the theory of $M_0$ as an $\infty$-categorical theory, and applying the monad-theory adjunction. To be more precise, assume we have $\mathcal{E}$ a locally presentable $\infty$-category, with $\mathcal{E}_O \subset \mathcal{E}$ a subcategory that is (equivalent to the nerve of) a locally presentable $1$-category. For example, $\mathcal{E}$ could be a category of presheaves of spaces on a $1$-category and $\mathcal{E}_0$ is the full subcategory of presheaves that are levelwise discrete (i.e. equivalent to presheaves of sets). If now $M$ is an ordinary monad on the $1$-category $\mathcal{E}_0$ which is $\mathcal{A}$-nervous for $\mathcal{A} \subset \mathcal{E}_0$ then, assuming $\mathcal{A}$ is also dense in $\mathcal{E}$, one can consider the monad on $\mathcal{E}$ associated by the monad-theory adjunction to the $\mathcal{A}$-theory of $M$. We will not develop this point of view much further, but many examples we mention in this paper can be thought as special case of this. The $E_1$ monad is obtained from the free monoid monad on Sets (as a subcategories of spaces). All the examples mentioned at the end of \cref{sec:chu_haugseng} can also be thought of as being obtained this way. The examples the monads for $E_2$ and $E_\infty$-algebras treated in \cref{sec:E1_algebra} can also be thought in this way, but with $\mathcal{E}_0$ and $\mathcal{A}$ being $2$-categories instead of $1$-categories.

\bigskip

We conclude this introduction by mentioning some closely related work:

Another approach to the monad-theory correspondence in $\infty$-categorical context has been developed very recently and independently from ours by R.~Kositsyn in \cite{kositsyn2021completeness}. Compared to our approach, Kositsyn uses more abstract methods, relying on the theory of $(\infty,2)$-categories. He also uses the description of monads as lax functors from the terminal category, while we take a more elementary approach following more closely Lurie's theory of monads from \cite{lurieHA}. Also, Kositsyn's work focuses on generalizing the notion of ``monads with arities'' from \cite{berger2012monads} (which we discuss in section \cref{sec:general_consequence}) while we consider the slightly more general notion of ``nervous monads'' from \cite{bourke2019monads}. While the gain in generality from using nervous monads instead of monads with arities is not essential by itself, it allows one to see the monad-theory equivalence as a special case of a more general monad-pretheory adjunction. The notion of pretheory is much simpler and has better category-theoretic properties than the various notion of theories considered. This makes pretheories much easier to handle when dealing with examples and is key in our construction in \cref{sec:general_consequence} of colimits of nervous monads and accessible monads on locally presentable categories, using colimits of pretheories.

In \cite{haugseng2020lax}, R.~Haugseng has developed a more general theory of monads in an $(\infty,2)$-category and proves it is equivalent to both Lurie's and Riehl-Verity's approach to monads (hence clarifying the equivalence between the two). We expect a large part of our preliminary results could be deduced from \cite{haugseng2020lax}. However, Haugseng's work relies on the some (as of yet unproven) assumptions about the Gray tensor products of $\infty$-categories, so we have decided to give independent and more elementary proof of these results.

Finally, our work is closely related to Chu and Haugseng's work on algebraic patterns from \cite{chu2019homotopy} and the precise relation is discussed in \cref{sec:chu_haugseng}. Essentially, algebraic patterns correspond to the special case of ``(pre)theories'' as above that represent parametric right adjoint cartesian monads (or polynomial monads in the terminology of \cite{chu2019homotopy}) on presheaf $\infty$-categories. Of course, it is not true that the results in \cite{chu2019homotopy} are all special cases of our results: parametric right adjoint cartesian monads have more structure than general monads and this translates into a better behaved theory in this special case.

\section{ Notation and preliminaries}

While we will try to give model independent argument whenever possible, we generally work within the framework of Jacob Lurie's books \cite{lurieHTT} and \cite{lurieHA}. An $\infty$-category is by definition a quasicategory, i.e. a simplicial set satisfying the appropriate lifting property. We refer to \cite{lurieHTT} for the basic theory of $\infty$-categories. We often will write objects (or 0-simplices) in an $\infty$-category by lower case letters, such as $x, y$. We call the 1-simplices of an $\infty$-category \emph{edges} or \emph{1-morphisms}. An edge is said to be an equivalence if and only if it represents an equivalence in the \emph{homotopy category} of an $\infty$-category (see \cite[Section 1.2.3]{lurieHTT} for the definition of the homotopy category). 

Given two objects $x, y$ in an $\infty$-category $\mathcal{C}$, we will write $\Map_{\mathcal{C}}(x, y)$ for the space of maps between $x$ and $y$. We will be working in a relatively model-independant manner, so it does not matter which of the (equivalent) models of mapping spaces from \cite[Section 1.2.2]{lurieHTT} we use. An \emph{equivalence of $\infty$-categories} is just an equivalence in Joyal's model structure for $\infty$-categories. That is, it induces an equivalence of homotopy categories, as well as  induces weak equivalences of mapping spaces. We will refer to fibrations in Joyal's model structure as \emph{quasi-fibrations}. Quasi-fibrations between quasi-categories have a nice characterization as \emph{isofibrations} (see \cite[Corollary 2.4.6.5]{lurieHTT}).

We will write $X^{K}$ for the internal hom in simplicial sets. If $X$ is an $\infty$-category, then $X^{K}$ is also an $\infty$-category and we write often write $\Fun(K, X)$, to emphasize that this is the \emph{$\infty$-category of functors} from $K$ to $X$. 
\\

By a \emph{simplicial category}, we mean a simplicially enriched category. 
Given a simplicial category $\mathcal{C}$, we will write $N(\mathcal{C})$ for its homotopy coherent nerve (see \cite[Definition 1.1.5]{lurieHTT}). It should be noted that in the case we regard an ordinary category as an enriched category with discrete mapping spaces, this recovers the ordinary nerve construction. 

Recall that a \emph{natural transformation} of maps of $\infty$-categories $f, g : \mathcal{C} \rightarrow \mathcal{D}$ is just a map $T : \mathcal{C} \times \Delta^{1} \rightarrow \mathcal{D}$ so that $T|_{\mathcal{C} \times \{ 0\}} = f, T|_{\mathcal{C} \times \{ 1\}} = g$. This is the same as a morphism in the functor $\infty$-category $\Fun(\mathcal{C},\mathcal{D})$.
 A natural transformation $T$ is a called a \emph{natural isomorphism} if corresponds to an invertible morphism in $\Fun(\mathcal{C}, \mathcal{D})$. We often write $T_{x} = T|_{\{x \} \times \Delta^{1}}$ which is an arrow in $f(x) \to g(x)$ in $\mathcal{D}$, and is called the \emph{component of $T$ at $x$}. We recall that:

\begin{lemma}\label{lem1.7}
Suppose that $T : \mathcal{C} \times \Delta^{1} \rightarrow \mathcal{D}$ is a natural transformation. The following are equivalent:
\begin{enumerate}
\item{$T$ is a natural isomorphism.}
\item{For each $x \in \mathcal{C}$, $T_{x}$ is an equivalence. }
\end{enumerate}

In other words, a natural transformation is a natural isomorphism iff each component is an equivalence. 
 
\end{lemma}

\begin{proof} This follows from \cite[Corollary 5.1.2.3]{lurieHTT} as an object $y$ is equivalent to an object $x$ in an $\infty$-category $\mathcal{C}$ iff $y$ is a (co)limit of $x : \Delta^{0} \rightarrow \mathcal{C}$. 
\end{proof}

We denote by $\Sp$ the $\infty$-category of spaces and by $\Prsh(\mathcal{C})$ the $\infty$-category of presheaves of spaces on an $\infty$-category $\mathcal{C}$, that is $\Prsh(\mathcal{C}) = \Fun(\mathcal{C}^{op},\Sp)$. We will write $y_{\mathcal{C}} : \mathcal{C} \rightarrow \Prsh(\mathcal{C})$ for the Yoneda embedding.

We refer the reader to \cite[Section 5.2.2]{lurieHTT} for the theory of adjoint functors, as well as related concepts such as counit transformations. 
In classical category theory, one can verify that functors form an adjoint pair by specifying the unit and counit of the adjunction, and verifying that they satisfy the triangle identities. The $\infty$-categorical counterpart of this statement, which follows, will be used several times throughout the paper:

\begin{lemma}\label{lem1.8}
Let $F : \mathcal{C} \rightarrow \mathcal{D}$, $G : \mathcal{D} \rightarrow \mathcal{C}$ be functors of $\infty$-categories. Let $\eta : id \rightarrow GF$ and $\epsilon : FG \rightarrow id$ be natural transformations. If for each object $X \in \mathcal{C}$ and $Y \in \mathcal{D}$  the two composites:

\[ F(X) \overset{F(\eta_X)}{\to} FGF(X) \overset{\epsilon_{F(X)}}{\to} F(X) \quad \text{and} \quad G(Y) \overset{\eta_{G(Y)}}{\to} GFG(Y) \overset{G(\epsilon_Y)}{\to} G(Y) \]

are equivalences, then $\eta$ is the unit of an adjunction $F \dashv G$.

\end{lemma}

By duality it is also the case that $\epsilon$ is the counit of an adjunction, but without additional assumption (for example the fact that the two composite above are equivalent to the identity) these two claims might not be compatible ($\eta$ and $\epsilon$ might not be the unit and counit of the same adjunction, typically, one of the adjunctions can be twisted by an automorphism of $F$ or $G$.)

\begin{proof}
By the definition of unit of an adjunction \cite[Proposition 5.2.2.7]{lurieHTT}, we want to show that for each $x \in \mathcal{C}, y \in \mathcal{D}$ the map 
\begin{equation}\label{unit2}
U_{x, y} : \Map_{\mathcal{D}}(Fx, y) \rightarrow \Map_{\mathcal{C}}(GFx, Gy) \xrightarrow{ (\uvar)  \circ \eta_{x}} \Map_{\mathcal{C}}(x, Gy)
\end{equation}
is an equivalence. We introduce the dual transformation

\[
V_{x, y} : \Map_\mathcal{C}(x, Gy) \rightarrow \Map_{\mathcal{D}}(Fx, FGy) \xrightarrow{\epsilon_{y} \circ (\uvar)} \Map_{\mathcal{D}}(Fx, y)
\]

Since the natural transformation $\epsilon$ and $\eta$ induces a natural tranformation on the level of enriched homotopy categories\footnote{Here we see the homotopy category as enriched in the homotopy category of spaces as in \cite[Definition 1.1.5.14]{lurieHTT}.}, we get a commutative square in the homotopy categry of spaces:

\[
\xymatrix
{
\Map_{\mathcal{C}}(x, G(y)) \ar[r]_<<<<{GF(-)} \ar[d]_{id} & \Map_{\mathcal{C}}(GF(x), GFG(y)) \ar[d]_{\eta_{x}} \\
\Map_{\mathcal{C}}(x, G(y)) \ar[r]_{\eta_{Gy} \circ (-)} & \Map_{\mathcal{C}} (x, GFG(y))
}
\]

In other words $GF(-) \circ \eta_{x} \simeq \eta_{G(y)} \circ (-)$.
We have \[ U_{x, y} \circ V_{x, y} = G(\epsilon_{y} \circ F(-)) \circ \eta_{x} = G(\epsilon_{y}) \circ GF(-) \circ \eta_{x} \simeq G(\epsilon_{y}) \circ \eta_{Gy} \circ (-) \]
so $U_{x, y} \circ V_{x, y} $ is the composition by an equivalence by our assumptions, hence $U_{x,y} \circ V_{x,y}$ is an equivalence. Similarly, we have that $V_{x, y} \circ U_{x, y} \simeq \epsilon_{y} \circ F(G(-) \circ \eta_{x}) = \epsilon_{y} \circ FG(-) \circ F(\eta_{x}) \simeq (-) \circ \epsilon_{Gx} \circ F(\eta_{x}) $, so $V_{x, y} \circ U_{x, y}$ is also an equivalence. It hence follows that $U_{x,y}$ and $V_{x,y}$ are both equivalences.  
\end{proof}

In Section \ref{sect:monads-theories}, we show that the monad-theory correspondence is an \emph{idempotent adjunction}. We will exploit the idempotence of the adjunction throughout the paper, especially in Section \ref{sec:E1_algebra}. Thus, we will review the definition and basic properties of idempotent adjunctions below:

\begin{lemma}
Suppose that $L \dashv R$ is an adjunction with counit $\epsilon$ and unit $\eta$. Then one of the following natural transformations $(\epsilon)L, R(\epsilon), \eta(R), L(\eta) $ is an equivalence if and only if each of them are equivalences. If any (and hence all) of the above natural transformations are equivalences, we say that the adjunction is \emph{idempotent}.
\end{lemma}

\begin{proof}
The classical, or 1-categorical, analogue of this fact is \cite[Proposition 2.8]{Towers}. The proof given there carries forward to the $\infty$-categorical case, either because it is essentially an excercise in manipulating the counit-unit identities, or be applying the 1-categorical result to the homotopy category and the adjunction between the derived functors of $L$ and $R$.
\end{proof}

\begin{remark}\label{idempobasics}
A useful fact about idempotent adjunctions is that the restrict to an equivalence $im(R) \simeq im(L)$ between the essential images of $R$ and $L$, essentially by definition. It is also important to note that if 
$X \in im(L), Y \in im(R)$, then also by definition $LRX \simeq X, Y \simeq RLY$.
\end{remark}

\begin{remark} Given an adjunction $L \dashv R$, written $L: \mathcal{C} \leftrightarrows \mathcal{D} : R$, post-composition with $L$ and $R$ induces an adjunction:

\[ (L \circ \uvar ) : \Fun(\mathcal{T},\mathcal{C}) \leftrightarrows \Fun(\mathcal{T},\mathcal{C}) : (R \circ \uvar) \]

for any $\infty$-category $\mathcal{T}$. A natural transformation $L X \to Y$ corresponds to a natural transformation $X \to RY$ simply by functoriality of the correspondence between arrows $L(a) \to b$ and arrows $a \to R(b)$.

But on the other hand, pre-composition with $L$ and $R$ induces an adunction in the other direction: 

\[ (\uvar \circ R ) : \Fun(\mathcal{D}, \mathcal{T}) \leftrightarrows \Fun(\mathcal{T},\mathcal{C}) : ( \uvar \circ L) \]

That is there is a correspondence between natural transformation $X \circ R \to Y$ and $X \to Y \circ L$. Indeed, given a natural transformation $v: X \to Y \circ L$, one obtain a natural transformation

\[ X R \overset{v R}{\to} Y LR \overset{Y(\eta)}{\to} Y \]

where $\eta: LR \to Id$ is the counit of adjunction. The inverse construction is obtained from the counit and the unit-counit relation shows that these are inverses of each other.

\end{remark}

 We refer to section 2.4 of \cite{lurieHTT} for the general theory of Cartesian and coCartesian fibrations. The following construction allows us to describe how the coCartesian fibration classified by $F:\mathcal{C} \to \icat$ relates to the coCartesian fibration classified by $\Fun(K,F( \uvar)) : \mathcal{C} \to \icat$ for a fixed $\infty$-category $K$: 

\begin{definition}\label{def:Fk} Let $\mathcal{E} \to \mathcal{B}$ be a map of simplicial sets and $K$ any simplicial set. We denote by $F_K(\mathcal{E})$ the simplicial set obtained as the pullback:

\[
\begin{tikzcd}
F_K(\mathcal{E}) \ar[dr,phantom,"\lrcorner"very near start] \ar[d] \ar[r] & \mathcal{E}^K \ar[d] \\
\mathcal{B} \ar[r] & \mathcal{B}^K, \\
\end{tikzcd}
\]
where the bottom map is the diagonal map.

\end{definition}







\begin{proposition}\label{prop:Fk}

\begin{enumerate}

\item If $\mathcal{E} \to \mathcal{B}$ is a Cartesian or coCartesian fibration, then $F_K \mathcal{E} \to \mathcal{B}$ is as well.

\item The construction $\mathcal{E} \mapsto F_K \mathcal{E}$ is right adjoint to $\mathcal{E} \mapsto \mathcal{E} \times K$ in the $\infty$-categories of Cartesian fibrations over $\mathcal{B}$ and of coCartesian fibrations over $\mathcal{B}$.

\item If $\mathcal{E} \to \mathcal{B}$ is a coCartesian fibration, then the functor $\mathcal{B} \to \icat$ classifying $F_K(\mathcal{E})$ is equivalent to the composite of the functor $\mathcal{B} \to \icat$ classifying $\mathcal{E} \to \mathcal{B}$ with $\Fun(\mathcal{K}, \uvar) : \icat \to \icat$.

\end{enumerate}

\end{proposition}

\begin{proof}

The first point for Cartesian fibrations follows immediately from Proposition 3.1.2.1 of \cite{lurieHTT}, which claims that $\mathcal{E}^K \to \mathcal{B}^K$ is a cartesian fibration when $\mathcal{E} \to \mathcal{B}$ is, and the fact that a pullback of a cartesian fibration is a cartesian fibration. The case of coCartesian fibrations immediately follows by duality. In order to prove the second point we will need to recall some element of the proof of Proposition 3.1.2.1 in \cite{lurieHTT}. 

The idea is that it is immediate to check that the construction $\mathcal{E} \mapsto F_K \mathcal{E}$ and $\mathcal{E} \mapsto K \times \mathcal{E}$ are a simplicially enriched pair of adjoint functors on the category (in the notation of \cite{lurieHTT}) $\sSet^+/\mathcal{B}^\sharp$ of marked simplicial sets over $\mathcal{B}^\sharp$ (which is $\mathcal{B}$ with all edges marked). The core result of section 3.1.2 of \cite{lurieHTT} is Proposition 3.1.2.3 which implies that product by $K$ preserves the ``marked anodyne maps''. This implies that the right adjoint $F_k(\uvar)$ preserves the objects with the right lifting property against these maps, i.e. exactly the Cartesian fibrations. However as taking the product with $K$ preserves the cofibrations, this pair of adjoint functors actually is a Quillen adjunction on the ``cartesian model structure'' (constructed in Proposition 3.1.3.7 of \cite{lurieHTT}) on $\sSet^+/\mathcal{B}^\sharp$. This implies that these functors induce an adjunction on the corresponding $\infty$-categories, which proves the second point for cartesian fibrations. The result for coCartesian fibrations follows by duality.

For the third point, while it is a bit difficult to keep track of what classifies the functor $F_K(\mathcal{E})$, it is relatively easy to observe that $K \times \mathcal{E} \to \mathcal{B}$ is classified by $K \times F(\uvar)$ where $F : \mathcal{B} \to \icat$ is the functor classifying $\mathcal{E} \to \mathcal{B}$. Indeed, by functoriality of the straightening/unstraightening construction in $\mathcal{B}$ one deduces that $\mathcal{B} \times K \to \mathcal{B}$ classifies the constant functor with value $K$, and one then uses that the straightening/unstraightening equivalence preserves products.

It follows that the right adjoint of these two constructions are also equivalent under the straightening/unstraightening equivalence. In the category of functors $\mathcal{B} \to \icat$, the right adjoint to $F \mapsto K \times F$ is indeed $F \mapsto \Fun(K,F(\uvar))$ and the second point above show that $F_k(\uvar)$ is the right adjoint of $\mathcal{E} \to K \times \mathcal{E}$. This concludes the proof.
  
\end{proof}

\section{Monads on $\infty$-Categories}
\label{sec:monads_prelim}

In the present paper, we follow Jacob Lurie's definition of monads on $\infty$-categories, from Chapter 4.7 of \cite{lurieHA}. In this section, we briefly recall some important points of Lurie's theory of monads and we complete the proof of \cref{th:monads=monadic} which claims that the category $\Mon_\mathcal{C}$ of monads in $\mathcal{C}$ is equivalent to the opposite of the full subcategory $\Md_\mathcal{C}$ of $(\icat)/\mathcal{C}$ of monadic right adjoint functors to $\mathcal{C}$. This result is mentioned without proof by Lurie in Remark 4.7.3.8 of \cite{lurieHA}.

Lurie's definition works as follows: given an $\infty$-category $\mathcal{C}$, he constructs a monoidal $\infty$-category of endofunctor $\End(\mathcal{C})$ that acts on $\mathcal{C}$. The category $\Mon_\mathcal{C}$ of monads on $\mathcal{C}$ is then defined as the category of monoids in $\End(\mathcal{C})$. As $\End(\mathcal{C})$ acts on $\mathcal{C}$, given a monad $T$ on $\mathcal{C}$ we can look at the category $\mathcal{C}^T$ of objects of $\mathcal{C}$ endowed with an action of $T$ (the left $T$-modules) and this is what we call the $\infty$-category of $T$-algebras, or the Eilenberg-Moore category of $T$.

In \cite{lurieHA} Lurie make sense of these notions of monoids and algebras (or rather modules in the general terminology) using his formalism of $\infty$-operads. In fact,\cite{lurieHA} developed two formalisms that allow one to do this: one can use the formalism of (symmetric) $\infty$-operads, or the formalism of planar (non-symmetric) $\infty$-operads. They are shown to be equivalent in \cite[Proposition 4.1.2.11]{lurieHA} and \cite[Theorem 2.3.3.23]{lurieHA}, but lead to different combinatorics for the concrete description of monads. Here we will recall all of the relevant definitions in the formalism of planar operads, in a way as unpacked as possible.

\begin{definition}\label{def:cartMonoid} A \emph{monoid object} $M$ in an $\infty$-category $\mathcal{C}$ with finite products is a functor $M: N(\Delta^{op}) \to \mathcal{C}$ which satisfies the Segal conditions:

\begin{itemize}

\item $M([0])$ is a terminal object of $\mathcal{C}$.

\item For each $n$, the map $M([n]) \to M([1])^n$, induced by the maps $[1] \simeq \{i,i+1\} \subset [n]$ for $i=0 \dots,n-1$ is an equivalence.

\end{itemize}

The category $\Monoid(\mathcal{C})$ of monoids in $\mathcal{C}$ is the full subcategory of $\mathcal{C}^{\Delta^{op}}$ on monoids.  $M([1])$ is called the underlying object of $M$.

\end{definition}

For example, if $M = M([1])$ is the underlying object of a monoid, the multiplication map $M^2 \to M$ is obtained as the map $M^2 \simeq M([2]) \to M([1])$ induced by $[1] \simeq \{0,2\} \subset \{0,1,2\}$. The associativity and higher coherence conditions are obtained by looking at the maps between the $M([k])$ for $k\geqslant 3$.

Note that this is the definition of monoid \emph{with respect to the cartesian product.} We will later give a definition of monoids with respect to a monoidal structure, which is different (they are equivalent when the monoidal structure is cartesian by (3) of \cite[Corollary 2.4.1.8]{lurieHA} and \cite[Proposition 2.4.2.5]{lurieHA}). The same remarks apply to the next definition as well:

\begin{definition}\label{def:cartModule} A module object in an $\infty$-category $\mathcal{C}$ with finite products, is a functor $X: N(\Delta^{op}) \times \Delta^1 \to \mathcal{C}$ such that:

\begin{itemize}

\item The restriction of $X$ to $N(\Delta^{op}) \times \{1\} \simeq N(\Delta^{op})$ is a monoid object in the sense of \cref{def:cartMonoid}.

\item The maps $X([n],0) \to X([n],1) \times X([0],0)$ induced by the maps $[0] \simeq \{n\} \subset [n]$ and obvious map $(0,[n]) \to (1,[n])$ are equivalences.

\end{itemize}
The $\infty$-category $\LMod(\mathcal{C})$ of modules is the full subcategory of functors $\mathcal{C}^{N(\Delta^{op}) \times \Delta^1}$ on module objects.
\end{definition}

The category $\LMod(\mathcal{C})$ should be thought of as a category of pairs of a monoid $M$ with an $M$-module $X$. The module $M$ is the restriction of $X$ to $N(\Delta^{op}) \times \{1\}$ which is a monoid by the first assumption. The ``underlying'' object $X$ is obtained as $X=X(0,[0])$, and the action map $M \times X \to X$ is induced by $X([1],0) \simeq X([1],1) \times X([0],0) = M \times X \to X([0],0)$ induced by the unique edge $[0] \to [1]$ in $N(\Delta)$.

This intuition that $\LMod(\mathcal{C})$ is a ``category of pairs'' is made formal by the following:

\begin{proposition} The forgetful functor from $\LMod(\mathcal{C}) \to \Monoid(\mathcal{C})$ that restricts to $N(\Delta^{op}) \times \{1\}$ is a Cartesian fibration. Its fiber over a monoid $T \in \Monoid(\mathcal{C})$ is called the category of $T$-modules and is denoted $\LMod^T(\mathcal{C})$. 
\end{proposition}

Henceforth, when we say that $X$ is an $M$-module we mean that $X$ is an object of $\LMod(\mathcal{C})$ over $M$. We call an action of $M$ on an object $X \in \mathcal{C}$ the data of a $M$-module whose underlying object is $X$.

This allows to defines a \emph{monoidal $\infty$-category} $\mathcal{M}$ to be a monoid in $\icat$. A \emph{monoidal action} of such a monoidal $\infty$-category $\mathcal{M}$ on an $\infty$-category $\mathcal{C}$ is an action in $\icat$ in the sense above.

We will generally work with monoidal $\infty$-categories and monoidal action from ``the other side'' of the straightening/unstraightening equivalences. Instead of defining a monoidal $\infty$-category as a functor $N(\Delta^{op}) \to \icat$, we define a monoidal $\infty$-category $\mathcal{M}$ to be a coCartesian fibration $\mathcal{M}^\circledast \to N(\Delta^{op})$ which is classified by a functor satisfying the Segal conditions as in \cref{def:cartMonoid}. Similarly, an action of $\mathcal{M}$ on an $\infty$-category $\mathcal{C}$ is defined as a coCartesian fibration $\mathcal{C}^\circledast \to N(\Delta^{op}) \times \Delta^1$ classified by a functor to $\icat$ satisfying the conditions of \cref{def:cartModule}.

The symbol $\circledast$ is only here to distinguish the underlying $\infty$-categories $\mathcal{M}$ and $\mathcal{X}$, which are the fiber over respectively $[1]$ and $([0],0)$, from the domain of these coCartesian fibrations.

\begin{remark}\label{rk:strict_monoid} If an $\infty$-category $\mathcal{M}$ has a monoid structure as a simplicial set, then it has a monoidal $\infty$-category. We call this a strict monoidal $\infty$-category. Indeed, one easily sees that such a ``strict monoidal'' $\infty$-category corresponds exactly to the functor $N(\Delta^{op}) \to \icat$, which comes from the $1$-categorical functor $\Delta^{op} \to \sSet$ that takes values in $\infty$-categories and satisfies the Segal condition up to isomorphism instead of just up to equivalence. Morphisms of simplicial monoids also induces monoidal functors.

Of course, the same can be said of a monoidal action. If $\mathcal{M}$ and $\mathcal{X}$ are two $\infty$-categories and $\mathcal{M}$ is a simplicial monoid acting on the simplicial set $\mathcal{X}$, then this produces a monoidal structure on $\mathcal{M}$ and a monoidal action of $\mathcal{M}$ on $\mathcal{X}$ in the sense above. The monoidal action can be encoded as functor $\Delta^{op} \times \Delta^1 \to \sSet$ that takes values in quasi-categories and satisfies the Segal conditions up to isomorphism.
\end{remark}

Next we move to the definition of monoids and monoidal actions in monoidal $\infty$-categories. We first need to introduce the following terminology:

\begin{definition}

\begin{itemize}

\item An edge in $N(\Delta^{op})$ is said to be \emph{inert} if the corresponding arrow in $\Delta$ is an interval inclusion, i.e. of the form $[k] \simeq \{i,i+1,\dots,i+k\} \subset [n]$ for $i+k \leqslant n$.

\item An inert edge in $N(\Delta^{op}) \times \Delta^1$ is a pair $(v,f)$ of an \emph{inert} edge $v$ (in the above sense) in $N(\Delta^{op})$ and an arbitrary edge $f$ in $\Delta^1$, such that if $f$ is the identity edge of $0$ then the map $v:[n] \to [m]$  satisfies $v(n)=m$.

\item If $X^\circledast \to N(\Delta^{op})$ is a monoidal $\infty$-category or a monoidal action, an arrow in $X^\circledast$ is said to be \emph{inert} if it is coCartesian and its image in $N(\Delta^{op})$ is inert.

\item If $X^\circledast \to N(\Delta^{op}) \times \Delta^1$ is a monoidal action, an arrow in $X^\circledast$ is said to be \emph{inert} if it is coCartesian and its image in $N(\Delta^{op}) \times \Delta^1$ is inert.

\end{itemize}
\end{definition}

Intuitively, the inert edges are the arrows in $N(\Delta^{op})$ or $N(\Delta^{op}) \times \Delta^1$ such that, given a monoid object $N(\Delta^{op}) \to \mathcal{C}$  or a module object $N(\Delta^{op}) \times \Delta^1 \to \mathcal{C}$ corresponds to product projection. A general arrow encodes some operations from the monoid or module structure.

We can now give the definition of monoids, monoid actions and module objects in a general monoidal $\infty$-category.

\begin{definition}
\begin{itemize}

\item If $\mathcal{C}^\circledast \to N(\Delta^{op})$ is a monoidal $\infty$-category, a \emph{monoid object} in $\mathcal{C}$ is a section of this map that send inert edges to inert edges. The $\infty$-category $\Monoid(\mathcal{C})$ is defined as the full subcategory of the $\infty$-category of sections on monoid objects.

\item If $\mathcal{X}^\circledast \to N(\Delta^{op}) \times \Delta^1$ is a monoidal action, a \emph{module object} in $\mathcal{X}$ is a section of this map that sends inert edges to inert edges. The $\infty$-category $\LMod(\mathcal{X})$ is defined as the full subcategory of the $\infty$-category of sections on module objects.

\end{itemize}

\end{definition}

Obviously, the notion of monoid in $\mathcal{C}$ depends on the whole monoidal structure $\mathcal{C}^\circledast \to N(\Delta^{op})$ and not just on the underlying $\infty$-category $\mathcal{C}$, and the notation $\Monoid(\mathcal{C})$ is an abuse. The same applies to module objects.

Here again, the monoidal action $\mathcal{X}^\circledast \to N(\Delta^{op}) \times \Delta^1$ is a pair of a monoidal $\infty$-category $\mathcal{M}$ that acts on an $\infty$-category $\mathcal{X}$. The category $\LMod(\mathcal{X})$ is a category of pairs of a monoid object $M$ in $\mathcal{M}$, together with an object $X$ of $\mathcal{X}$ and an action of $M$ on $X$.

We sometime write $\LMod(\mathcal{X},\mathcal{M})$ when we want to emphasize the monoidal part of the action $\mathcal{X}^\circledast \to N(\Delta^{op}) \times \Delta^1$.

Similar to the case of $\infty$-categories with finite limits, if $\mathcal{X}$ is an $\infty$-category with an action of a monoidal $\infty$-category $\mathcal{M}$, then there is a forgetful functor $\LMod(\mathcal{X}) \to \Monoid(\mathcal{M})$ and Lurie showed that this is a cartesian fibration. If $A$ is a monoid object in $\mathcal{M}$ we denote by $\LMod^A(\mathcal{X})$ the fibre over $A$ of this fibration. We call it the category of $A$-modules in $\mathcal{X}$.The full subcategory whose objects are actions of $A$ on $B \in \mathcal{X}$ is denoted by $\LMod^A_{B}(\mathcal{X})$.

Before moving further, we quickly look at how these notions interact with the functions $F_K$ of \cref{def:Fk}. Let $\mathcal{M}^\circledast \to N(\Delta^{op})$ be a monoidal $\infty$-category and $\mathcal{X}^\circledast \to N(\Delta^{op}) \times \Delta^1$ a monoidal action of $\mathcal{M}$ on an $\infty$-category $\mathcal{X}$. For $K$ an $\infty$-category, we can apply the construction $F_K$ of \cref{def:Fk} to these functors to get new functors $F_K \mathcal{M}^\circledast \to N(\Delta^{op})$ and $F_K \mathcal{X}^\circledast \to N(\Delta^{op}) \times \Delta^1$. We have:

\begin{lemma}\label{lemma:Fun_is_monoidal} $F_K \mathcal{M}^\circledast \to N(\Delta^{op})$ and $F_K \mathcal{X}^\circledast \to N(\Delta^{op}) \times \Delta^1$ are a monoidal $\infty$-category and a monoidal action. They correspond, respectively, to a monoidal structure on $\Fun(K,\mathcal{M})$ and a monoidal action of $\Fun(K,\mathcal{M})$ on $\Fun(K,\mathcal{X})$.\end{lemma}

\begin{proof}
 By \cref{prop:Fk} these are coCartesian fibration classified by the postcomposition of the functor classifying $\mathcal{M}^\circledast$ and $\mathcal{X}^\circledast$ with $\Fun(K,\uvar)$. As $\Fun(K,\uvar)$ preserves products, it is immediate that the corresponding functors to $\icat$ satisfies the ``Segal conditions'' of \cref{def:cartMonoid} and \cref{def:cartModule}. This immediately proves the result.
\end{proof}

\begin{lemma}\label{lem:alg_in_fun} We have natural equivalences (in fact isomorphisms) of $\infty$-categories:

\[
\begin{tikzcd}
   \LMod( F_K \mathcal{X}^\circledast) \ar[d] \ar[r,phantom,"\simeq"description] & \Fun(K,\LMod(\mathcal{X})) \ar[d] \\
   \Monoid( F_K \mathcal{M}^\circledast) \ar[r,phantom,"\simeq"description] & \Fun(K,\Monoid(\mathcal{M}))
\end{tikzcd}
\] 

compatible to the forgetful functor as represented in the diagram above.
\end{lemma}

\begin{proof} By construction of $F_K$, or rather by the second point of \cref{prop:Fk}, the simplicial set of sections of $F_K \mathcal{X}^\circledast \to N(\Delta^{op}) \times \Delta^1$ is equivalent to the simplicial set of maps $K \times N(\Delta^{op}) \times \Delta^1 \to \mathcal{X}^\circledast$. This, in turn, is isomorphic to the simplicial set of maps from $K$ to the simplicial set of sections of $\mathcal{X}^\circledast \to N(\Delta^{op}) \times \Delta^1$. The same can be said for $\mathcal{M}^\circledast \to N(\Delta^{op})$, and these identification are compatible with the ``forgetful functors'', i.e. the restriction along $N(\Delta^{op}) \times \{1\} \to N(\Delta^{op}) \times \Delta^1$.

The $\infty$-categories mentioned in the lemma are full subcategories of these simplicial sets. To conclude the proof we just need to show that they are preserved by these isomorphisms. The proofs for monoids and module objects are exactly the same. On the side of $\LMod( F_K \mathcal{X}^\circledast)$ we are looking at the full subcategory of sections that send any inert arrow to a coCartesian lift. Though the series of isomorphisms mentioned at the beginning, these corresponds to the dotted section in
\[
\begin{tikzcd}
& \Fun(K, \mathcal{X}^\circledast) \ar[d] \\
  N(\Delta^{op}) \times \Delta^1 \ar[ur,dotted] \ar[r] &   \Fun(K, N(\Delta^{op}) \times \Delta^1) 
\end{tikzcd}
\]
that sends inert edges to coCartesian edges. The coCartesian edges with respect to the coCartesian fibration $\Fun(K, \mathcal{X}^\circledast) \to \Fun(K, N(\Delta^{op}) \times \Delta^1) $ are exactly the natural transformations that are coCartesian when evaluated at each object $k \in K$ (see \cite[Proposition 3.1.2.1]{lurieHTT}). Thus, it follows that through the series of isomorphisms above, a section of $F_K \mathcal{X}^\circledast  \to N(\Delta^{op}) \times \Delta^1 $ corresponds to a module object if and only the corresponding functor from $K$ to the simplicial set of section of $\mathcal{X}^\circledast  \to N(\Delta^{op}) \times \Delta^1 $ sends each object of $k \in K$ to a module object. This concludes the proof.

\end{proof}

\begin{lemma}\label{lem:diag_is_monoidal} If $\mathcal{M}$ is a monoidal $\infty$-category and $K$ any $\infty$-category, then the diagonal functor $\mathcal{M} \to \Fun(K,\mathcal{M})$ admits a structure of monoidal functor. \end{lemma}

\begin{proof} This follows immediately from the fact that $\mathcal{M} \to \Fun(K, \mathcal{M})$ is natural in $\mathcal{M}$ and that the monoidal structure on  $\Fun(K,\mathcal{M})$ is obtained by postcomposing the functor $N(\Delta^{op}) \to \icat$ classifying the monoidal structure of $\mathcal{M}$ with $\Fun(K,\uvar)$.\end{proof}

\begin{remark}\label{rmkx} We fix $\mathcal{M}$ a monoidal $\infty$-category with an action on an $\infty$-category $\mathcal{X}$, and $K$ any $\infty$-category. For $M$ any monoid object in $\mathcal{M}$, one can use the monoidal functor of \cref{lem:diag_is_monoidal} to see $M$ as a ``constant'' monoid object in $\Fun(K,\mathcal{M})$. Through the monoidal action of $\Fun(K,\mathcal{M})$ on $\Fun(K,\mathcal{X})$ introduced by \cref{lemma:Fun_is_monoidal}, we can look at the $\infty$-category

\[ \LMod^M(\Fun(K,\mathcal{X})) \]

of $M$-modules in $\Fun(K,\mathcal{X})$. We then have, as a special case of \cref{lem:alg_in_fun} an equivalence (in fact an isomorphism)

\[ \LMod^M(\Fun(K,\mathcal{X})) \simeq \Fun(K,\LMod^M(\mathcal{X})). \]

Indeed, the left hand side corresponds to the fiber of $\LMod(\Fun(K,\mathcal{X})) \simeq \Fun(K,\LMod(\mathcal{X}))$ over $M \in \Fun(K,\Monoid(\mathcal{M}))$. However, given that $M$ is in $\Monoid(\mathcal{M})$ this actually is a fiber of $F_K(\LMod(\mathcal{X}))$, and hence can be identified with the simplicial set of functors from $K$ to the fiber of $\LMod(\mathcal{X})$ as explained in \cref{prop:Fk}. This also shows that these equivalences are natural in $M$.

\end{remark}

We will write $\End(\mathcal{C})$ for the simplicial monoid of endomorphisms of an $\infty$-category $C$. By \ref{rk:strict_monoid}, it has the structure of a monoidal $\infty$-category. In \cite{lurieHA} Lurie defines the \emph{category of monads on $\mathcal{C}$}, which we denote by $\Mon_{\mathcal{C}}$, to be the category of monoid objects in $\End(\mathcal{C})$. Given a monad $M \in \Mon_{\mathcal{C}}$ acting on a category $\mathcal{E}$, and a monad $T$ on $\mathcal{C}$, we write $\mathcal{E}^{T}$ for the category of $T$-modules.

\begin{construction}\label{cstr:EndC_actson_FunDC} Let $\mathcal{C}$ and $\mathcal{D}$ be two $\infty$-categories. In \cite{lurieHA}, Lurie construct an action of $\End(\mathcal{C})$ on $\Fun(\mathcal{D},\mathcal{C})$ by looking at the strict action of the simplicial monoid $\End(\mathcal{C})$ on the simplicial set $\Fun(\mathcal{D},\mathcal{C})$. 

This is however equivalent to the construction we discussed above by combining the action of $\Fun(\mathcal{D},\End(\mathcal{C}))$ on $\Fun(\mathcal{D},\mathcal{D})$ obtained from \cref{lemma:Fun_is_monoidal} and the monoidal functor $\End(\mathcal{C}) \to \Fun(\mathcal{D},\End(\mathcal{C}))$ from \cref{lem:diag_is_monoidal}.

Indeed, we start from the strict action of $\End(\mathcal{C})$ on $\mathcal{C}$, which can be encoded by a functor $\Delta^{op} \times \Delta^1 \to \sSet$ as discussed in \cref{rk:strict_monoid}, and our construction in \cref{lemma:Fun_is_monoidal} using $F_K$ is known (by \cref{prop:Fk}) to be equivalent to post-composing this functor by $\Fun(K,\uvar)$. But this is precisely the strict action considered in the first paragraph.

\end{construction}

From the discussion of \ref{rmkx} and \ref{lem:alg_in_fun} above we obtain

\begin{lemma}\label{lem:Action_on_functor} The natural functor

\[  \Fun(K,\mathcal{C})^T \to \Fun(K,\mathcal{C}^T ) \]

is an equivalence of $\infty$-categories, compatible to the forgetful functor to $\Fun(K,\mathcal{C})$.
\end{lemma}

The final ingredient to Lurie's theory of monads is the notion of \emph{endomorphism object}. Given a monoidal $\infty$-category $\mathcal{C}$ acting on an $\infty$-category $\mathcal{X}$ and $X \in \mathcal{X}$ any object, Lurie considers the $\infty$-category $\mathcal{C}[X]$ which can informally be described as the $\infty$-category of object $Y \in \mathcal{C}$ endowed with a map $Y \otimes X \to X$ in $\mathcal{X}$ (see Definition 4.7.1.1 in \cite{lurieHA} for a more formal statement).

\begin{definition}\label{def:end_object} Let $\mathcal{C}$ be a monoidal $\infty$-category and $\mathcal{X}$ an $\infty$-category with an action of $\mathcal{C}$. An \emph{endomorphism object} for an object $X \in \mathcal{X}$ is (if it exists) a terminal object in the category $\mathcal{C}[X]$.
\end{definition}

As usual, we will, in an abuse of language, say that an object $\EndO(X) \in \mathcal{C}$ is an endomorphisms object of $X$ if it is the image of a terminal object in $\mathcal{C}[X]$ by the forgetful functor $\mathcal{C}[X] \to \mathcal{C}$. Lurie also shows in \cite[Remark 4.7.1.33 and Proposition 4.7.1.34]{lurieHA} that:

\begin{proposition} In the situation above, the $\infty$-category $\mathcal{C}[X]$ admits a monoidal structure for which the forgetful functor $\mathcal{C}[X] \to \mathcal{C}$ is monoidal.
\end{proposition}

\begin{proposition}\label{prop:end_action} Given $\mathcal{C}$ a monoidal $\infty$-category and $\mathcal{X}$ an $\infty$-category with an action of $\mathcal{C}$, if $X \in \mathcal{X}$ admits an endomorphisms object $\EndO(X) \in \mathcal{C}$, then 
$\EndO(X)$ is a monoid object, it acts on $X$, and we have equivalences $\Map_{\Monoid(\mathcal{C})}(B, \EndO(X)) \simeq \LMod_{B}^{X}(\mathcal{X})$, natural in $B \in \Monoid(\mathcal{C})$.
\end{proposition}

Note that the identity arrow $\EndO(X) \to \EndO(X)$ in particular corresponds to an action of the monoid $\EndO(X)$ on $X$ which we call the canonical action of $\EndO(X)$ on $X$.

\begin{proof} 
The equivalence is essentially that of \cite[Corollary 4.7.1.41]{lurieHA}, which is deduced from \cite[Corollary 4.7.1.40]{lurieHA}. However, we should note that \cite[Corollary 4.7.1.41]{lurieHA} do not explicitly claims that this equivalence is natural in $B$ (only that it is ``canonical''). It seems that the naturality of the equivalence is implicit, and is later implicitly used in the rest of Section 4.7 of \cite{lurieHA}. For this reason, we decided to explain some key points of the proof from section 4.7.1 of \cite{lurieHA} and especially clarify how the naturality follows.

A first remark is that Lurie introduces an alternative model for $\mathcal{C}[X]$, more precisely he constructs a monoidal $\infty$-category $\mathcal{C}^{+}[X]$ for each $X \in \mathcal{X}$, such that there is a trivial fibration $\mathcal{C}^{+}[X] \rightarrow \mathcal{C}[X]$ and which has slightly better properties than $\mathcal{C}[X]$. By examining the proof of \cite[Corollary 4.7.1.40]{lurieHA}, the equivalence comes from a string of equivalences 
\begin{equation}\label{moduleequiv}
\Monoid(\mathcal{C})_{/\EndO(X)} \leftarrow \Monoid(\mathcal{C}^{+}[X])_{/T_{X}} \rightarrow \Monoid(\mathcal{C}^{+}[X]) \rightarrow \LMod^{X}(\mathcal{X}),
\end{equation}
where $T_{X}$ is a terminal object of $\mathcal{C}^{+}[X]$ whose image in $\mathcal{C}$ is $\EndO(X)$. The fact that such an object exists exactly translates to the assumption that $X$ admits an endomorphism object $\EndO(X)$. As a terminal object of the monoidal $\infty$-category $\mathcal{C}^{+}[X]$, it follows from Corollary 3.2.2.5 and Proposition 4.1.3.19 of \cite{lurieHA} that $T_X$ has a monoid structure that makes it a terminal object of $\Monoid(\mathcal{C}^{+}[X])$. The monoid structure on $\EndO(X)$ is obtained from the one on $T_X$ as the functor $\mathcal{C}^{+}[X] \to \mathcal{C}$ is monoidal.

The theorem is deduced from these equivalences and the fact that all the categories involved admits right fibrations to $\Monoid(\mathcal{C})$ and all functors in (\ref{moduleequiv}) are compatible (up to equality) to these fibrations. Hence taking the fibers over a monoid $B \in \Monoid(\mathcal{C})$ in the zig-zag of equivalence (\ref{moduleequiv}) gives a series of equivalences:

\begin{equation}\label{moduleequiv2}
\Map_{\Monoid(\mathcal{C})}(B, \EndO(X)) \leftarrow (\Monoid(\mathcal{C}^{+}[X])_B)_{/T_{X}} \rightarrow \Monoid(\mathcal{C}^{+}[X])_B \rightarrow \LMod^{X}_{B}(\mathcal{X})
\end{equation}

where the $B$ index denotes fiber over $B$. The (contravariant) functoriality in $B$ of these all these constructions and the naturality of these equivalence hence follows immediately from the straightening construction.

The functor $\Monoid(\mathcal{C})_{/\EndO(X)} \to \Monoid(\mathcal{C})$ is the obvious forgetful functor and is hence a right fibration (by the dual of \cite[Corollary 2.1.2.2]{lurieHTT}). The functor $\theta : \mathcal{C}^{+}[X] \rightarrow \mathcal{C}$  constructed in \cite[Proposition 4.7.1.39]{lurieHA} induces a right fibration $\Monoid(\mathcal{C}^+[X]) \to \Monoid(\mathcal{C})$ (also by \cite[Proposition 4.7.1.39]{lurieHA}). As $T_X$ is sent to $\EndO(X)$ by this functor, this induces a right fibration $\Monoid(\mathcal{C}^+[X])_{/T_X} \to \Monoid(\mathcal{C})_{/\EndO(X)}$. This clearly equips the first three categories with right fibrations to $\Monoid(\mathcal{C})$ with the first two functor being compatible to these (by functoriality of the slice construction).

The functor $\LMod^{X}(\mathcal{X}) \to \Monoid(\mathcal{C})$ is simply the composite of the functor $\LMod^{X}(\mathcal{X}) \to \LMod(\mathcal{X})$ with the forgetful functor $\LMod(\mathcal{X}) \to \Monoid(\mathcal{C})$, it can be seen as the top of arrow in the pullback:

\[
\begin{tikzcd}
 \LMod^X(\mathcal{X}) \ar[dr,phantom,"\lrcorner"very near start] \ar[d]  \ar[r] & \{X\} \times \Monoid(\mathcal{C}) \ar[d] \\
\LMod(\mathcal{X}) \ar[r] & \mathcal{X} \times \Monoid(\mathcal{C})
\end{tikzcd}
\]

Given that the bottom map is an iso-fibration, it follows that $\LMod^{X}(\mathcal{X}) \to \Monoid(\mathcal{C})$ is a quasi-fibration. The fact that it is a right fibration will be deduced later from the equivalence with the right fibration $\Monoid(\mathcal{C}^+[X]) \to \Monoid(\mathcal{C})$ (see \cite[Corollary 4.7.1.42]{lurieHA}).

So, if we consider the diagram: 
$$
\xymatrix
{
\Monoid(C^{+}[X])  \ar[r] \ar[dr]_{\theta'} &  \LMod^{X}(\mathcal{X}) \ar[d] \\
& \Monoid(\mathcal{C})
}
$$
where the diagonal map is the map $\theta'$ mentioned above (whose fibre over $B$ is $\Monoid(\mathcal{C}^{+}[X])_B$), the horizontal map is the equivalence of \cite[Theorem 4.7.1.34]{lurieHA}, and the vertical map is the forgetful functor, which is a cartesian fibration. One can then check from the explicit construction of the horizontal map given in \cite{lurieHA} that the above diagram commutes, since all functors involved are induced by 'forgetful functors' between various full subcategories of functor categories from (nerve of) $1$-categories. Hence producing the last compatibility we needed.
\end{proof}

\begin{remark}\label{EndoAmbiguity} Consider the $\infty$-category $\icat$ of all $\infty$-categories with the usual cartesian monoidal structure. Then for any $\infty$-category $\mathcal{C} \in \icat$, its endomorphism object $\EndO(\mathcal{C})$ is just the $\infty$-category of endofunctors of $\mathcal{C}$, and  \cref{prop:end_action} makes it into a monoidal $\infty$-category acting on $\mathcal{C}$. Though in this case given that $\End(\mathcal{C})$ can simply be concretely defined as the simplicial monoid of maps $\mathcal{C} \to \mathcal{C}$ one can also obtain this monoidal structure in much more explicit way from its strictly associative monoid structure. It is fairly easy to check that the two descriptions are equivalent.
\end{remark}

Using the action of $\End(\mathcal{C})$ on $\Fun(\mathcal{D},\mathcal{C})$ mentioned in \cref{cstr:EndC_actson_FunDC}, we can specialize the notion of endomorphism object to the notion of endomorphisms monads. Following Definition 4.7.3.2 of \cite{lurieHA} we have:

\begin{definition}\label{def:Lurie_endo_monad} An \emph{endomorphism monad} $T$ for a functor $U:\mathcal{D} \to \mathcal{C}$ is a monad $T \in \Mon(\mathcal{C})=\Monoid(\End \mathcal{C})$ with an action of $T$ on $F$ such that the action map $T U \to U$ identify $T$ as an endomorphism object for $U$. \end{definition}

\begin{remark}\label{rk:def_endo_monad} Let $U : \mathcal{D} \to \mathcal{C}$ be a functor that admits an endomorphism object $\EndO(U) \in \End(\mathcal{C})$, for the action of $\End(\mathcal{C})$ on $\Fun(\mathcal{D},\mathcal{C})$ from \cref{cstr:EndC_actson_FunDC}. By \cref{prop:end_action},  $\EndO(U)$ gets a monoid (i.e. monad) structure, and a canonical action of $\EndO(U)$ on $U$, obtained from the identity map of $\EndO(U)$ through the equivalence of \cref{prop:end_action}. This monad $\EndO(U)$, with its action on $U$, is then an endomorphisms monad for $U$ in the sense of \cref{def:Lurie_endo_monad}, and any endomorphism monad is of this form (in an essentially unique way).

Indeed, saying that $T$ is a monad acting on $U$, means, by \cite[Theorem 4.7.1.34]{lurieHA}, that when we use the action map $T U \to T$ to identifies $T$ with an object of $\End(\mathcal{C})[U]$ it has a monoid structure. Now, as \cref{def:Lurie_endo_monad} asks for $T$, endowed with this map $TU \to U$, be a terminal object in $\End(\mathcal{C})[U]$ (by \cref{def:end_object}), this monoid structure is essentially unique. and makes $T$ into the terminal monoid in $\End(\mathcal{C})[U]$. 

Now, the action of $\EndO(U)$ on $U$ we mentioned is obtained from the identity of $\EndO(U)$ through the equivalence of categories $\Monoid(\End(\mathcal{C}))_{\EndO(U)} \simeq \Monoid(\End(\mathcal{C})[U])$. Since the identity is terminal in the slice category, it corresponds to a terminal object of $\Monoid(\End(\mathcal{C})[U])$, so that both description boils down to ``terminal objects in $\Monoid(\End(\mathcal{C})[U])$''.

Given this, we will denote $\EndO(U)$ the endomorphism monad of $U$ if it exists.

\end{remark}

Lemma 4.7.3.1 of \cite{lurieHA} describes the endomorphism monads of right adjoint functor in the usual way: 

\begin{proposition} If $U : \mathcal{D} \to \mathcal{C}$ is a functor with a left adjoint $F$, then $U \circ F:\mathcal{C} \to \mathcal{C}$ endowed with the map $U \circ F \circ U \to U$ given by applying $U$ to the unit of adjunction is an endomorphisms monad for $U$.
\end{proposition}

 We can construct a functor $\Mon_\mathcal{C}^{op} \to \icat$ that sends $T$ to $\mathcal{C}^T$ by applying straightening to the Cartesian fibration $\LMod(\End(\mathcal{C})) \rightarrow \Monoid(\End(\mathcal{C}))$  associated to the action in \cref{cstr:EndC_actson_FunDC}.

\begin{proposition}\label{prop:Monad_forget_adjunction} The functor

\[ \begin{array}{ccc}
(\Mon_\mathcal{C})^{op} & \to & (\icat)_{/ \mathcal{C}} \\
T & \mapsto & \mathcal{C}^T
\end{array}
\]

Corestricted to the full subcategory of right adjoint functors admit a left adjoint that sends a right adjoint functor $U : \mathcal{D} \to \mathcal{C}$ to its endomorphism monad. 

\end{proposition}

\begin{proof}

To show the existence of the adjoint, it suffices to show that the functor $T \mapsto \Map_{(\icat)_/\mathcal{C}}(\mathcal{D},\mathcal{C}^T) $ is representable by $\EndO(U)$.
By applying \ref{prop:end_action} to the action of $\End(\mathcal{C})$ on $(\icat)_{/\mathcal{C}}$ given by \ref{cstr:EndC_actson_FunDC}, and applying \ref{lem:Action_on_functor}, we get equivalences (natural in $T$)

\[ \Map_{\Mon_\mathcal{C}}(T, \EndO(U)) \simeq  \Map^U_{\icat}(\mathcal{D},\mathcal{C})^{T} \simeq \Map^U_{\icat}(\mathcal{D},\mathcal{C}^T)  \]

where $\Map^U_{\icat}(\mathcal{D},\mathcal{C})^{T}$ and $\Map^U_{\icat}(\mathcal{D},\mathcal{C}^T)$ are the (homotopy) fibers of $\Map_{\icat}(\mathcal{D},\mathcal{C})^{T}$and  $\Map_{\icat}(\mathcal{D},\mathcal{C}^T)$ over $U \in \Map_{\icat}(\mathcal{D},\mathcal{C})$. By the description of mapping spaces in a slice $\infty$-category from \cite[Proposition 5.5.5.12]{lurieHTT}, one has an equivalence $\Map^U_{\icat}(\mathcal{D},\mathcal{C}^T) \simeq \Map_{(\icat)_{/\mathcal{C}}}(\mathcal{D},\mathcal{C}^T)$, which in total gives an equivalence natural in $T$:

\[ \Map_{\Mon_\mathcal{C}}(T, \EndO(U)) \simeq \Map_{(\icat)_{/\mathcal{C}}}(\mathcal{D},\mathcal{C}^T) \]

\end{proof}

\begin{lemma}
Let $U : \mathcal{D} \rightarrow \mathcal{C}$ be a functor of $\infty$-categories. The unit of the adjunction of \cref{prop:Monad_forget_adjunction} can be identified with the canonical map $\mathcal{D} \rightarrow \mathcal{C}^{\EndO(U)}$ determined by the action of $\EndO(U)$ on $U$, through the equivalence $\Fun(\mathcal{D},\mathcal{C}^{\EndO(U)}) \simeq \Fun(\mathcal{D},\mathcal{C})^{\EndO(U)}$ of \cref{lem:Action_on_functor}.
\end{lemma}

\begin{proof}

We need to chase through the series of equivalences in the proof of \cref{prop:Monad_forget_adjunction} the image of $id:\EndO(U) \to \EndO(U)$ in  $\Map_{(\icat)_{/\mathcal{C}}}(\mathcal{D},\mathcal{C}^{\EndO(U)})$.

The first step of this series of equivalences

\[ \Map_{\Mon_\mathcal{C}}(T, \EndO(U)) \simeq  \Map^U_{\icat}(\mathcal{D},\mathcal{C})^{T} \]

sends the identity of $\EndO(U)$ to the canonical action of $\EndO(U)$ on $U$ (see \cref{rk:def_endo_monad}), essentially by definition of this action. The map to $\Map_{(\icat)_{/\mathcal{C}}}(\mathcal{D},\mathcal{C}^T)$ is then essentially just the isomorphism $\Fun(\mathcal{D},\mathcal{C}^{\EndO(U)}) \simeq \Fun(\mathcal{D},\mathcal{C})^{\EndO(U)}$, hence the result. 

\end{proof}

A right adjoint functor $U:\mathcal{E} \to \mathcal{C}$ is said to be \emph{monadic} if the unit of adjunction $\mathcal{E} \to \mathcal{C}^{\EndO(U)}$ is an equivalence.

Theorem~4.7.3.5 of \cite{lurieHA} is an $\infty$-categorical version of the Barr-Beck theorem. It states that a right adjoint functor $U: \mathcal{E} \to \mathcal{C}$ is monadic if and only it is conservative and for every simplicial object in $\mathcal{E}$ whose image by $U$ is split has a colimit which is preserved by $U$.

Given that forgetful functors of the form $\mathcal{C}^T \to \mathcal{C}$ themselves satisfy all these conditions, this shows that the adjunction of \cref{prop:Monad_forget_adjunction} is an idempotent, and identifies the category $\Mon_\mathcal{C}$ of monads on a category $\mathcal{C}$ with the opposite of the category of monadic right adjoint functor $\mathcal{E} \to \mathcal{C}$, seen as a full subcategory of $(\icat)_{/\mathcal{C}}$. In particular, one deduces:

\begin{theorem}\label{th:monads=monadic} For any $\infty$-category $\mathcal{C}$, the functor 

\[ \begin{array}{ccc}
(\Mon_\mathcal{C})^{op} & \to & (\icat)_{/ \mathcal{C}} \\
T & \mapsto & \mathcal{C}^T
\end{array}
\]

is fully faithful and identifies $(\Mon_\mathcal{C})^{op}$ with $\Md_\mathcal{C}$ the reflective full subcategory of $(\icat)_{/ \mathcal{C}}$ of monadic right adjoint functors.

\end{theorem}

This result was alluded to in Remark 4.7.3.8 of \cite{lurieHA}, but wasn't proven.

We finish with a consequence of Lurie's Barr-Beck theorem that will be useful in a few places:

\begin{proposition} \label{prop:Pullback_monadic} Given a (homotopy) pullback square of $\infty$-categories:

\[
\begin{tikzcd}
  \mathcal{D}' \ar[r,"G"] \ar[d,"V"] \ar[dr,phantom,"\lrcorner"very near start] & \mathcal{D} \ar[d,"U"] \\
\mathcal{C}' \ar[r,"F"] & \mathcal{C} \\
\end{tikzcd}
\]

if $U$ is a monadic right adjoint functor and $V$ is a right adjoint functor then $V$ is monadic.

\end{proposition}

\begin{proof} We show that if $U$ satisfies the conditions of Lurie's Barr-Beck monadicity theorem (i.e. Theorem 4.7.3.5 of \cite{lurieHTT}), then so does $V$.

An arrow $f \in \mathcal{D}'$ is invertible if and only if both its image and $\mathcal{C}'$ and $\mathcal{D}$ are invertible. But if its image in $\mathcal{C}'$ is invertible, then its image in $\mathcal{C}$ is as well. Hence, as $U$ is conservative, its image in $\mathcal{D}$ is also invertible. Thus, $V$ is conservative. 
  
Let $X:\Delta \rightarrow \mathcal{D}'$ be a $V$-split simplicial diagram. Its image in $\mathcal{D}$ is a $U$-split simplicial diagram, hence it admit a colimit which is preserved by $U$. The colimit of $X$ in $\mathcal{C}'$ is split, and is thus preserved by $F$, since split colimits are preserved by all functors (\cite[Lemma  6.1.3.16]{lurieHTT}). It follows that $X$ has a colimit both in $\mathcal{D}$ and $\mathcal{C}'$ which is preserved by $U$ and $F$. Hence, it has a colimit in $\mathcal{D}'$ which is preserved by both projections by the lemma below.
\end{proof}

\begin{lemma}\label{colimitinternal}
Suppose that we have a diagram
\[
\xymatrix
{
N(I)^{\triangleleft} \ar[r]_{\phi} & \mathcal{D} \ar[r] \ar[d] & \ar[d]_{g} \mathcal{X} \\
& \ar[r]_{f} \mathcal{Y} & \mathcal{Z}
}
\]
where the square is a homotopy pullback square of $\infty$-categories and $I$ is any category. Suppose that $\phi$ determines a colimit diagram in $\mathcal{X}, \mathcal{Y}, \mathcal{Z}$. Then $\phi$ is a colimit diagram in $\mathcal{D}$. 
 
\end{lemma}

\begin{proof}
Because of the Quillen equivalence between Bergner's model structure on simplicial categories and Joyal's structure, we can replace the above diagram with the nerve of a diagram of (fibrant) simplicial categories. By \cite[4.2.4.1]{lurieHTT}, we thus reduce to the corresponding statement about simplicial categories, where the homotopy pullback is taken with respect to Bergner's model structure. For each pair of objects $x, y \in \mathcal{Y}, x', y' \in \mathcal{X}$ such that $f(x) = g(x'), f(y) = g(y')$, we have a homotopy pullback:
\[
\xymatrix
{
\Map_{\mathcal{D}}((x, x'), (y, y')) \ar[d] \ar[r] & \Map_{\mathcal{X}}(x, y) \ar[d] \\
\Map_{\mathcal{Y}}(x, y) \ar[r] & \Map_{\mathcal{Z}}(f(x), f(y))
}
\]

This follows from the construction of homotopy pullbacks in Bergner's model structure. The result now follows from the description of homotopy colimits internal to a fibrant simplicial category (\cite[Remark A.3.3.13]{lurieHTT}) and the fact that homotopy pullbacks and homotopy colimits of simplicial sets commute (see \cite[6.1.3.14]{lurieHTT}.

\end{proof}

Finally, we will need the following lemma that is essentially a consequence of \cref{th:monads=monadic}:

\begin{lemma}\label{cor:equiv_of_monaidc} Let $U_1 : \mathcal{D}_1 \to \mathcal{C}$ and $U_2 : \mathcal{D}_2 \to \mathcal{C}$ be two monadic right adjoint functors, with left adjoints $L_1$ and $L_2$ and $t: \mathcal{D}_1 \to \mathcal{D}_2$ be a functor such that $U_1 \simeq U_2 t$. Then $t$ is an equivalence of $\infty$-categories if and only if the natural transformation $L_2 \to t L_1$ obtained from the isomorphism $U_1 \to U_2 t$ through the adjunction is an equivalence.\end{lemma}

\begin{proof} Under the equivalence \cref{th:monads=monadic}, $t$ corresponds to a morphisms of monads $\End(U_2) \to \End(U_1)$, and $t$ is an equivalence if and only if this morphism of monads is an equivalence. At the level of underlying endofunctors, the morphism of monads identifies with a natural transformation $U_2 L_2 \to U_1 L_1$ induced by the action of $U_2L_2$ on $U_1 \simeq U_2 \circ t$. Thus, it can be described as the natural transformation $U_2 L_2 \to U_1 L_1 \simeq U_2 t L_1$ obtained under the adjunction $L_1 \dashv U_1$ from the map $U_2 L_2 U_2 t \to U_2 t$ induced by the counit $L_2 U_2 \to Id$.

Unfolding this, we see that up canonical isomorphism, this map $U_2 L_2 \to U_1 L_1$ is exactly the image under $U_2$ of the natural transformation $L_2 \to t L_1$. As $U_2$ is conservative it indeed follows that the morphism of monads is an equivalence if and only if $L_2 \to t L_1$ is an equivalence.
  
\end{proof}

\begin{remark} In the rest of the paper, we will never use explicitly use the notion of monads, but always work with monads through the equivalence of \cref{th:monads=monadic}. The only exception to this is \cref{cor:equiv_of_monaidc} that will be used in the proof of \cref{thm5.1}.




In particular, any theory of monads for which \cref{th:monads=monadic} and \cref{cor:equiv_of_monaidc} are valid can be used instead of Lurie's theory of monads. We suspect this should apply for example to Riehl-Verity theory of monads on $\infty$-categories from \cite{riehl2016homotopy}.
\end{remark}

\section{Partial adjoints and functoriality of the Kleisli category}
\label{sec:Partial_adjoints}

\begin{definition} If $T$ is a monad on an $\infty$-category $\mathcal{C}$, we denote by $\mathcal{C}_T$ the full subcategory of the $\infty$-category $\mathcal{C}^T$ of $T$-algebras on free $T$-algebras. That is, those $T$-algebras in the essential image of the free $T$-algebra functor $\mathcal{C} \to \mathcal{C}^T$.
$\mathcal{C}_T$ is called the \textit{Kleisli category} of $\mathcal{C}$.
\end{definition}

As the title suggests, the goal of this section is to study the functoriality properties of the construction $T \mapsto \mathcal{C}_T$. While $T \mapsto \mathcal{C}^T$ has a contravariant functoriality, the Kleisli category has a covariant functoriality essentially given by taking the left adjoint $f_!$ to $f^*$ for $f: T \to M$ a morphism of monads. However (even in ordinary category theory) the existence of a left adjoint $f_! \dashv f^*$ is in general not guaranteed, and when it exists its construction generally requires a complicated transfinite construction or an application of the special adjoint functor theorem.  In particular, given that we have not proven at this point that the $\infty$-category of algebras $\mathcal{C}^T$ has colimits or is a presentable category it would not be reasonable to assume that such a left adjoint exists. Instead we need to consider $f_!$ as a ``partial left adjoint'' in the following sense:

\begin{definition} Let $R: \mathcal{C} \to \mathcal{D}$ be a functor between $\infty$-categories. Let $\mathcal{D'} \subset \mathcal{D}$ be a full subcategory. One says that $R$ has a \emph{partial left adjoint} on $\mathcal{D'}$ if for all $X \in \mathcal{D'}$, the functor:

 \[ \begin{array}{rcl}
\mathcal{C} & \to & \Sp \\
Y & \mapsto & \Map_{\mathcal{D}}(X,R(Y))
\end{array} \]

is representable. If $\mathcal{C'} \subset \mathcal{C}$ is a full subcategory of $\mathcal{C}$, one says that $R$ has a partial left adjoint from $\mathcal{D'} \to \mathcal{C'}$ if for all $X \in \mathcal{D'}$ the object $Y$ as above is in $\mathcal{C'}$. We define \emph{partial right adjoint} in the dual way.

\end{definition}

By the $\infty$-categorical Yoneda lemma, it follows that when $R$ has a partial left adjoint $\mathcal{D'} \to \mathcal{C'}$ then there is an essentially unique functor $F:\mathcal{D'} \to \mathcal{C}$, called the partial left adjoint of $R$, endowed with an adjunction isomorphism:

\[ \Map_{\mathcal{D}}(X,R(Y)) \simeq \Map_{\mathcal{C}}(F(X),Y)\]

natural in $X \in \mathcal{D'}$ and $Y \in \mathcal{C}$.

As mentioned above, our main example of partial left adjoints comes from morphisms of monads:

\begin{proposition}\label{prop:f!_exists} Let $f : T \to M$ be a morphism of monads on a category $\mathcal{C}$. Then the forgetful functor between their categories of algebras $f^*: \mathcal{C}^M \to \mathcal{C}^T$ has a partial left adjoint $f_! : \mathcal{C}_T \to \mathcal{C}_M$ between the full subcategories $\mathcal{C}_T \subset \mathcal{C}^T$ and $\mathcal{C}_M \subset \mathcal{C}^M$ of free algebras.
\end{proposition}

\begin{proof} Let $U:\mathcal{C}^T \to \mathcal{C}$ and $V:\mathcal{C}^M \to \mathcal{C}$ be the two forgetful functors.

For any free algebra $X=T(A) \in \mathcal{C}_T$ and $Y$ an $M$-algebra, we have a series of isomorphisms all natural in $Y \in \mathcal{C}^M$:
\[ \Map_{\mathcal{C}^T}(X,f^* Y ) \simeq \Map_{\mathcal{C}}(A,U(f^* Y) ) \simeq \Map_{\mathcal{C}}(A,V( Y) ) \simeq \Map_{\mathcal{C}^M}(M A, Y ) . \]

Thus, the functor $\Map_{\mathcal{C}^T}(X,f^* \uvar )$ is representable by $MA$, which concludes the proof. 
\end{proof}

In order to study the functoriality properties of the Kleisli category construction, we will consider more generally the question of how partial left adjoints assemble into a $\icat$-valued functor. This occurs in exactly the same way as left adjoints assemble into a $\icat$-valued functor (as show for example for adjointable functors between locally presentable $\infty$-categories in \cite[Corollary 5.5.3.4]{lurieHTT}). To remind ourselves of the main case of interest, i.e. the category of monads, we will use similar notation for the general case:

\begin{assumption}\label{assumption:partial_adjoints_func} Consider a functor $\mathcal{D}^{op} \to \icat$, denoted $d \mapsto X^d$. For $f: d \to d'$ an arrow in $\mathcal{D}$, we denote the induced functor by $f^* : X^{d'} \to X^{d}$.

We also assume that for each object $d \in \mathcal{D}$, we have a full subcategory $X_d \subset X^d$ such that for each edge $f:d \to d'$, $f^*:X^{d'} \to X^d$ has a partial left adjoint $f_! : X_d \to X_{d'}$.

It should be noted that this automatically implies that if $d$ and $d'$ are isomorphic in $\mathcal{D}$, then the subcategory $X_d$ and $X_d'$ are identified by the equivalence between $X^d$ and $X^{d'}$. 
\end{assumption}

\begin{proposition} \label{prop:Partial_adjoint_fct} Let $X^\bullet : \mathcal{D}^{op} \to \icat$ be a functor as in \cref{assumption:partial_adjoints_func} above. Then there is a functor $\mathcal{D} \to \icat$ that sends each object of $d$ to $X_d$ and each arrow $f$ to $f_!$.
\end{proposition}

A precise construction of the functor is given in the proof and will be important on a few occasions in the rest of the paper.

\begin{proof} Let $\pi: \mathcal{X} \to \mathcal{D}$ be the cartesian fibration classified by $X$. Up to equivalence of $\infty$-categories one can freely assume that objects of $\mathcal{X}$ are pairs $(d,x)$ where $d$ is an object of $\mathcal{D}$ and $x$ is an object of $\mathcal{X}^d$.

We write $\mathcal{X'}$  for the full subcategory of $\mathcal{X}$ of objects of the form $(d,x)$ for $x \in \mathcal{X}_d$, and we claim that $\mathcal{X'} \to \mathcal{D}$ is a cocartesian fibration classifying a functor as described in the proposition.

Indeed, for each arrow $f:d' \to d$ and $x \in X_{d'}$, we have a unit arrow $x \to f^*f_!x$ in $X^{d'}$ constructed from the adjunction isomorphism in the usual way. It corresponds to an arrow $(d',x) \to (d,f_! x)$ in $\mathcal{X}$. Exactly as in the case of actual adjunction (see the proof of ``$(2) \Rightarrow (1)$'' of Proposition 5.2.2.8 of \cite{lurieHTT}), the adjunction isomorphism shows that this arrow is a locally $\pi$-cocartesian arrow in $\mathcal{X}$.

And Corollary 5.2.2.4 of \cite{lurieHTT} shows that, as $\pi$ is a Cartesian fibration, any locally $\pi$-coCartesian arrow is actually coCartesian, so this construction provide us with coCartesian lifts of any arrow $d' \to d$ for any object in $\mathcal{X'}$ over $d'$.

By the definition of $\mathcal{X'}$ its fiber over an object $d\in \mathcal{D}$ is indeed equivalent to $X_d$, and the way we constructed the cocartesian lift shows the functoriality is exactly the $f_!$ functor.

\end{proof}

It immediately follows from \cref{prop:f!_exists} and \cref{prop:Partial_adjoint_fct} that:

\begin{corollary}\label{cor:functoriality_of_Kleisli} The Kleisli category construction $T \mapsto \mathcal{C}_T$ defines a functor $\Mon_\mathcal{C} \to \icat$. Each morphism of monads$f:T \to M$ is sent to the partial left adjoint $f_! : \mathcal{C}_T \to \mathcal{C}_M$ to $f^*$.
\end{corollary}

\begin{remark} Because the initial object of $\Mon_{\mathcal{C}}$ is the identity monad $I$ and the Kleilsli category $\mathcal{C}_I$ of $I$ is equivalent to $\mathcal{C}$, it immediately follows that the Klesli category construction can actually be seen as a functor from $\Mon_{\mathcal{C}}$ to the coslice category $(\icat)_{\backslash \mathcal{C}}$, sending each monad $T$ to the free algebra functor $\mathcal{C} \to \mathcal{C}_T$.
\end{remark}

\begin{proposition} \label{prop:partial_adj_natural} Let $X^\bullet$ and $Y^\bullet$ be two functors $\mathcal{D}^{op} \to \icat$ as in \cref{assumption:partial_adjoints_func}. Let $\lambda: X^\bullet \to Y^\bullet$ be a natural transformation between them such that:
  \begin{enumerate}
  \item\label{prop:partial_adj_natural:ass1} For each object $d \in \mathcal{D}$, the functor $\lambda(d) : X^d \to Y^d$ sends $X_d$ to $Y_d$.
   \item\label{prop:partial_adj_natural:ass2} For each morphism $f:d' \to d$ in $\mathcal{D}$, the natural transformation $\lambda(d) f_! \to f_! \lambda(d')$ obtained from the naturality square $\lambda(d') f^* \overset{\sim}{\to} f^* \lambda(d)$ through the partial adjunction between $f_!$ and $f^*$, is an isomorphism.

  \end{enumerate}

Then, there is a natural transformation $\lambda' : X_\bullet \to Y_\bullet$ between the functors $\mathcal{D} \to \icat$ constructed in \cref{prop:Partial_adjoint_fct}, which on objects is the restriction of $\lambda$ and whose naturality isomorphism is the natural isomorphism $\lambda(d) f_! \to f_! \lambda(d')$ mentioned above.

\end{proposition}

\begin{proof} Let $\mathcal{X}, \mathcal{Y} \to \mathcal{D}$ be the cartesian fibrations corresponding to $X,Y : \mathcal{D}^{op} \to \icat$. And let $\mathcal{X'}, \mathcal{Y'} \to \mathcal{D}$ be the cocartesian fibration constructed in the proof of \cref{prop:Partial_adjoint_fct}.

By functoriality of the Grothendieck (or unstraightening) construction, the natural transformation $\lambda$ induces a functor $V : \mathcal{X} \to \mathcal{Y}$ in $(\icat)_{/\mathcal{D}}$ that preserves cartesian arrows. Assumption \ref{prop:partial_adj_natural:ass1}, immediately shows that $V$ restricts to a functor $\mathcal{X'} \to \mathcal{Y'}$ (also in $(\icat)_{/\mathcal{D}}$). Assumption \ref{prop:partial_adj_natural:ass2} translates to the fact that this functor sends cocartesian arrows to cocartesian arrows. Indeed, by uniqueness of cocartesian lifts, any cocartesian arrow in $\mathcal{X}$ is up to equivalence an arrow $(d,x) \to (d',f_!x)$ over $f:d \to d' \in \mathcal{D}$ corresponding to the unit of adjunction $x \to f^* f_! x$ as in the proof of \cref{prop:Partial_adjoint_fct}, for $x \in X_d$. The functor $V$ sends such an arrow to the arrow $(d,\lambda^d(x)) \to (d, \lambda^{d'} f_! x)$. This in turn corresponds to $\lambda^{d} x \to f^*\lambda^{d'} f_! x$ which is the image of the co-unit $x \to f^* f_! x$ under $\lambda^d$ up to the isomorphism $\lambda^d f^* \simeq f^* \lambda^{d'}$. Under assumption $(2)$, this maps identifies with the counit $\lambda^d(x) \to f^* f_! \lambda^d(x)$ and hence corresponds to a cocartesian arrow of $\mathcal{Y'}$.

As $V$ preserves cocartesian arrows from $\mathcal{X'}$ to $\mathcal{Y'}$, it corresponds to a natural transformation between the functors constructed in \cref{prop:Partial_adjoint_fct} with the properties claimed in the proposition.
\end{proof}

\begin{proposition}\label{yonedanat} Let $X^\bullet : \mathcal{D}^{op} \to \icat$ be a functor with subcategories $X_\bullet$ as in \cref{prop:Partial_adjoint_fct}. Then there are natural transformation:

\[ (\mathcal{X}_d)^{op} \to \Fun(\mathcal{X}^d,\Sp)  \]
\[ \mathcal{X}^d \to \Fun(\mathcal{X}_d^{op},\Sp)  \]

which are levelwise the restriction of the Yoneda embeddings. Here, $X_d$ has its covariant functoriality from \cref{prop:Partial_adjoint_fct}, $X^d$ has its original contravariant functoriality and we use the contravariant functoriality of $\Fun(\uvar,\Sp)$ given by restriction of presheaves to make the right hand side into functors with the appropriate variance.
\end{proposition}

\begin{proof}

$\Fun(\uvar, \Sp)$ has two different functorialities. Firstly, it has the natural contravariant functoriality used in the statement of the proposition, where each induced map $f^{*} : \Fun(\mathcal{X}^d,\Sp) \rightarrow \Fun(\mathcal{X}^{d'},\Sp)$ induced by $f : X^{d'} \rightarrow X^{d}$ has a right adjoint. The second functoriality is then given by applying \cref{prop:Partial_adjoint_fct} to obtain a covariant functoriality $\mathcal{C} \mapsto \Fun(\mathcal{C},\Sp)$, where morphisms acts as the left adjoint to the reindexing functors given by the contravariant functoriality. It was shown in section 6 of \cite{hebestreit2020orthofibrations} that the Yoneda embeddings $\mathcal{C} \to \Prsh(\mathcal{C})$ can be made into a natural transformation when $\Prsh(\mathcal{C}) = \Fun(\mathcal{C}^{op},\Sp)$ is endowed with this second functoriality.

In particular, we have a natural transformation $(\mathcal{X}^{d})^{op} \to \Fun(\mathcal{X}^d,\Sp)$, or equivalently $\mathcal{X}^d \to \Fun(\mathcal{X}^d,\Sp)^{op}$ where on the right hand side $\Fun(\uvar,\Sp)$ has its covariant (i.e. left adjoint) functoriality.

One can then apply \cref{prop:Partial_adjoint_fct} to $\mathcal{X}_d \subset (\mathcal{X}^{d})$ to recover the covariant functoriality of $\mathcal{X}_d$ (given by the $(f_!)^{op}$) and to $d \mapsto \Fun(\mathcal{X}^d,\Sp)^{op}$ to recover its usual ``precomposition'' functoriality as in the proposition. Hence, \cref{prop:partial_adj_natural} shows that the Yoneda embedding can be assembled into  a natural transformation

\[ (\mathcal{X}_d) \to \Fun(\mathcal{X}^d,\Sp)^{op}.  \]

The first condition \ref{prop:partial_adj_natural:ass1} is vacuous in this case given that the subcategories used on the right  hand side are the whole category, and the second condition is easy to check. Indeed, the natural transformation between the left adjoint coming from the naturality square along a map $f:d \to d' \in \mathcal{D}$ is, for each $X \in \mathcal{X}_d$, the map in $(\Fun(\mathcal{X}^{d'},\Sp))^{op}$, which, when evaluated on a $Y \in \mathcal{X}^{d'}$ is the map

\[\Map(f_!(X),Y) \to \Map(X,f^*(Y)) \]

obtained by applying the $f^*$ functoriality and precomposing with the unit $X \to f^* f_! X$. But essentially by definition, this map is an equivalence.

Taking opposite categories on both sides gives us the first natural transformation mentioned in the proposition:

\[ \mathcal{X}_d^{op} \to \Fun(\mathcal{X}^d,\Sp), \]

which is levelwise given by the restriction of the Yoneda embedding. The second one can be obtained formally from the first ones: informally, a natural transformation $(\mathcal{X}_d)^{op} \to \Fun(\mathcal{X}^d,\Sp)$ can be seen as a dinatural transformation $(\mathcal{X}_d)^{op} \times \mathcal{X}^d \to \Sp$. This, in turn, can be seen as a natural transformation $\mathcal{X}^d \to \Fun(\mathcal{X}_d^{op},\Sp)$ which is the second one. To avoid the use of dinatural transformations in this argument (which to the authors' knowledge have not been formalized in the $\infty$-categorical framework), one can use Proposition 5.1 of \cite{gepner2017lax} or Proposition 2.3 of \cite{glasman2016spectrum}. These assert that for any pairs of functors $F,G : \mathcal{C} \to \mathcal{D}$ the space of natural transformation from $F$ to $G$ can be described as the end\footnote{The end of a functor $\mathcal{C} \times \mathcal{C}^{op} \to \mathcal{D}$ is the limit indexed by the twisted arrow category Tw$(\mathcal{C}) \to \mathcal{C} \times \mathcal{C}^{op}$. See \cite{gepner2017lax} or \cite{glasman2016spectrum} }:

\[ \Map(F,G) \simeq \int_{c \in \mathcal{C}} \Map(F(c),G(c)). \]

In both cases a natural transformation $\lambda :F \to G$ corresponds to an element of the end whose component in $\Map(F(c),G(c))$ is simply $\lambda_c : F(c) \to G(c)$.

Using this (and the functoriality of ends) we have isomorphisms:

\[ \int_{d \in \mathcal{D}} \Fun(\mathcal{X}_d^{op}, \Fun(\mathcal{X}^d,\Sp)) \simeq \int_{d \in \mathcal{D}} \Fun(\mathcal{X}_d^{op} \times \mathcal{X}^d,\Sp) \]
\[ \simeq \int_{d \in \mathcal{D}} \Fun(\mathcal{X}^d, \Fun(\mathcal{X}_d^{op},\Sp)). \]

Through these isomorphisms, we hence obtain a natural transformation $\mathcal{X}^d \to \Fun(\mathcal{X}_d^{op},\Sp)$ that for each $d$ is given by the restricted Yoneda embedding.

\end{proof}

Applying this to the $\infty$-category of monads, we obtain:

\begin{corollary}\label{cor:rest_yon_fct} The restricted Yoneda embeddings $\mathcal{C}^T \to \Prsh(\mathcal{C}_T)$ can be equipped with the structure of a natural transformation between functors $(\Mon_{\mathcal{C}})^{op} \to \icat$.\end{corollary}

\section{The Monad-Theory Correspondence}\label{sect:monads-theories}

Throughout this section, we fix a locally presentable $\infty$-category $\mathcal{E}$, as well as a \emph{dense, small, full subcategory} $\mathcal{A} \subset \mathcal{E}$. 

We write $\PTh_{\mathcal{A}}$ for the full subcategory of $(\icat)_{\mathcal{A}/ }$ of essentially surjective functors $\mathcal{A} \rightarrow \mathcal{K}$ (with $\mathcal{K}$ also being small). Objects of $\PTh_\mathcal{A}$ are called $\mathcal{A}$-\textit{pretheories}.

\begin{definition}\label{def:Model} For a $\mathcal{A}$-pretheory $\mathcal{K}$, we define the category of $\mathcal{K}$-models as the pullback:

\[
\begin{tikzcd}
\Mod_\mathcal{E}(\mathcal{K}) \ar[r] \ar[dr,phantom,"\lrcorner"very near start] \ar[d] & \Prsh(\mathcal{K}) \ar[d] \\
\mathcal{E} \ar[r] & \Prsh(\mathcal{A}),
  \end{tikzcd}
\]
where the right vertical arrow is the restriction functor and the bottom horizontal arrow is the restricted Yoneda embedding, or ``$\mathcal{A}$-nerve'' functor. That is, it is the composite of the Yoneda embedding $\mathcal{E} \to \Prsh(\mathcal{E})$ with the restriction to $\mathcal{A} \subset \mathcal{E}$.

\end{definition}

\begin{proposition}\label{Prop:Monadic} The forgetful functor $\Mod_\mathcal{E}(\mathcal{K}) \to \mathcal{E}$ is a monadic right adjoint functor. The functor $\Mod_\mathcal{E}(\mathcal{K}) \to \Prsh(\mathcal{K})$ is a fully faithful right adjoint (i.e. is an equivalence to the inclusion of a reflective subcategory). \end{proposition}

\begin{proof} The functor $\Prsh(\mathcal{K}) \to \Prsh(\mathcal{A})$ is a monadic right adjoint functor. Indeed, it is conservative because $\mathcal{A} \to \mathcal{K}$ is essentially surjective. It satisfies the condition on split simplicial diagrams because it preserves all colimits and both $\Prsh(\mathcal{K})$ and $\Prsh(\mathcal{A})$ have all colimits.

Moreover, by Theorem 5.5.3.18 of \cite{lurieHTT}, the above can be seen as a pullback in the category of presentable $\infty$-categories and accessible right adjoint functors, hence the functors $\Mod_\mathcal{E}(\mathcal{K}) \to \mathcal{E}$ and $\Mod_\mathcal{E}(\mathcal{K}) \to \Prsh(\mathcal{K})$ are both right adjoint functors.

The monadicity of the first one then follows from \cref{prop:Pullback_monadic} and the second one is fully faithful since it is the pullback of $\mathcal{E} \to \Prsh(\mathcal{A})$ which is fully faithful as $\mathcal{A}$ is dense in $\mathcal{E}$.\end{proof}

\begin{construction}
The functoriality of the pullback in \cref{def:Model} and the contravariant functoriality of $\mathcal{K} \mapsto \Prsh(\mathcal{K})$, make $\Mod_\mathcal{E}(\uvar)$ into a functor $\PTh_{\mathcal{A}}^{op} \rightarrow (\icat)_{ / \mathcal{E}}$. By using the identification of \ref{th:monads=monadic} and taking opposite categories, we obtain a functor:

 \[ \begin{array}{rcl} \PTh_{\mathcal{A}} & \rightarrow & \Mon_{\mathcal{E}} \\
\mathcal{K} & \mapsto & \Mt^\mathcal{K},
\end{array}
\] 

which is characterized by the natural isomorphism $\mathcal{E}^{\Mt^{\mathcal{K}}}  \simeq \Mod_\mathcal{E}(\mathcal{K})$.

\end{construction}

\begin{lemma}\label{slicetopreth}
There is a functor $(\icat)_{\mathcal{A} /} \to \PTh_\mathcal{A}$ which takes each arrow $\mathcal{A} \to \mathcal{X}$ to its essential image $\mathcal{A} \to \mathcal{Y} \subset \mathcal{X}$.
\end{lemma}

\begin{proof}
We claim that in $\icat$ essentially surjective functors and fully faithful functors form an orthogonal factorization system (in the sense of \cite[Definition 5.2.8.8]{lurieHTT}). The result then follows from \cite[Lemma 5.5.8.19]{lurieHTT}. 

Indeed, this is just the (-1)-connected case of the n-connected/n-truncated factorization which exists in any locally presentable $\infty$-category by Proposition 4.6 of \cite{Univalence-Kock}. $\icat$ can be presented as the simplicial category of bifibrant objects of the variant of the Joyal model structure on marked simplicial sets (from \cite[Proposition 3.1.3.7]{lurieHTT} in the special case where $S= \Delta[0]$), which is a simplicial combinatorial model category, so $\icat$ is a locally presentable $\infty$-category by \cite[Theorem A.3.7.6]{lurieHTT}, and the factorization system exists.
\end{proof}

\begin{definition}\label{def1.6}
 Let $\mathrm{Th} : \Mon_{\mathcal{E}} \rightarrow \PTh_{\mathcal{A}}$ be the composite 
\[
\Mon_{\mathcal{E}} \xrightarrow{\mathcal{E}_{\bullet}} (\icat)_{\mathcal{E} /} \xrightarrow{ (-) \circ i} (\icat)_{\mathcal{A} /} \to \PTh_\mathcal{A}
\]
where the first functor is the Kleisli category functor constructed in \cref{cor:functoriality_of_Kleisli} and the last functor is the functor from \ref{slicetopreth} that takes the fullyfaithful-essentially surjective factorization.
\end{definition}

As shown in \ref{lem1.8}, to produce an adjunction of $\infty$-categories, it suffices to produce a counit and unit transformation, and verify the triangle identities on components. We will apply this strategy to show that $\Mt^{(-)} \dashv \Th$. 
\\

\begin{construction}\label{con1.9}
Consider the commutative square from \cref{def:Model}. By taking the left adjoint of each functor,
we get a commutative diagram in $(\mathrm{Cat}_{\infty})$:
\begin{equation}\label{uniteq}
\xymatrix
{
\mathcal{E}^{\Mt^{\mathcal{K}}}  && \ar[ll] \mathrm{Pr}(\mathcal{K}) & \ar[l]_{y_{\mathcal{K}}} \mathcal{K}  \\
\ar[u] \mathcal{E} && \ar[u] \ar[ll] \mathrm{Pr}(\mathcal{A}). & 
}
\end{equation}
By taking the essential image of the top horizontal composite we get a map $\eta_{\mathcal{K}} : \mathcal{K} \rightarrow \mathrm{Th}(\Mt^{\mathcal{K}})$. 
Since we can view \cref{def:Model} as lying in the $\infty$-category of locally presentable $\infty$-categories and accessible functors (\cite[Definition 5.5.3.1]{lurieHTT}) the operation of taking adjoints is functorial in $\mathcal{K}$ (\cite[Corollary 5.5.3.4]{lurieHTT}).

Essentially surjective functors and faithful functors form an orthogonal factorization system \ on $\icat$ (see  \cref{slicetopreth}). Thus, the operation of taking essential image is functorial by \cite[Lemma 5.2.8.19]{lurieHTT}, so $\eta_\mathcal{K}$ is natural in $\mathcal{K}$. This will be the unit of our adjunction. \end{construction}

\begin{construction}\label{con1.10}

We have a diagram natural in $M$
\begin{equation}\label{counit}
\xymatrix
{
\mathcal{E}^{M} \ar@{.>}[dr]^{\epsilon'_{M}} \ar@/_1pc/[ddr] \ar@/^1pc/[drr] & & \\
& \mathcal{E}^{\Mt^{\Th(M)}} \ar[r] \ar[d] & \Prsh(\Th(M)) \ar[d] \\
 & \mathcal{E} \ar[r]_{i^{*}} & \Prsh(\mathcal{A})
}
\end{equation}

The Yoneda functoriality described in \cref{yonedanat} gives us the naturality of the outer square, and the inner square is just \cref{def:Model}. $\epsilon'_{M}$ comes from the universal property of pullback and is hence (contravariantly) natural in $M$. Through the contravariant equivalence of \cref{th:monads=monadic} this corresponds to a natural transformation $\epsilon_M : \mu^{\Th(M)} \to M$, which will be the counit our the monad-theory adjunction.
\end{construction}

\begin{lemma}\label{lem1.11}
$\eta \circ \Th$ and $\Th \circ \epsilon$ are both natural equivalences.  
\end{lemma}

\begin{proof}

By \cref{lem1.7} to show that $\eta \circ \Th$ and $ \Th \circ \epsilon$ are natural equivalences, it suffices to show that for each monad $M$, the functors $\eta_{\Th(M)}$ and $\Th(\epsilon_{M})$ are equivalences. We will first show that $\eta_{\Th(M)} \circ \Th(\epsilon_{M})$ is an equivalence. Then we will show that each $\eta_{\Th(M)}$ is an equivalence, from which the required results will follow.  

Given a pretheory $\mathcal{K}$, we write $G_{\mathcal{K}} : \mathcal{E}^{\Mt^{\mathcal{K}}} \rightarrow \Pr(\mathcal{K})$ for the top horizontal map in the pullback of \cref{def:Model}. 
We write $Y_{M} : \mathcal{E}^{M} \rightarrow \Pr(\Th(M))$ for the restricted Yoneda embedding. $Y_{M}$ restricts to an equivalence $S : \Th(M) \simeq im(y_{\Th(M)})$, and the homotopy inverse $\Psi : im(y_{\Th(M)}) \rightarrow \Th(M)$ of $S$ is partial left adjoint of the map $Y_{M}$. Consider the commutative diagram (which is part of the diagram (\ref{counit})):
 $$
\xymatrix
{
\mathcal{E}^{M} \ar[d]_{\epsilon'_{M}} \ar[dr]^{Y_{M}} & \\
\mathcal{E}^{\Mt^{\Th(M)}} \ar[r]_{G_{\Th(M)}} & \Pr(\Th(M)).
&
}
$$
As noted in \cref{Prop:Monadic} $G_{\Th(M)}$ is a fully faithful right adjoint. We write $(G_{\Th(M)})^{L}$ for its left adjoint.
By the functoriality of taking partial left adjoints established in Section \ref{sec:Partial_adjoints} we have that $\Th(\epsilon_{M}) \circ (G_{\Th(M)})^{L}|_{im(y_{\Th(M)})} \simeq \Psi$ (note that $\Th(\epsilon_{M})$ is a partial left adjoint to $\epsilon'_{M}$ by construction). Let $\psi' \psi$ be the factorization of $y_{\Th(M)}$ through its essential image. Since $y_{\Th(M)}$ is fully faithful, $\psi$ is an equivalence. We have that $\eta_{\Th(M)} = (G_{\Th(M)})^{L}|_{im(y_{\Th(M)})} \circ \psi$. Thus, $\Psi \circ \psi = \Th(\epsilon_{M}) \circ \eta_{\Th(M)}$ is an equivalence.

We want to show now that $\eta_{\Th(M)}$ is an equivalence. It is essentially surjective by construction. We want to show that it induces a bijection on homotopy groups of mapping spaces. It induces a monomorphism of homotopy groups of mapping spaces since it has a left inverse. 

As noted in \cref{Prop:Monadic} $G_{\Th(M)}$ is fully faithful, so we have $G_{\Th(M)}^{L} \circ G_{\Th(M)} \simeq id$. $G_{\Th(M)}^{L}$ induces a surjection on homotopy groups for each mapping spaces between objects in the image of $G_{\Th(M)}$. The essential image of the restricted Yoneda embedding in (\ref{counit}) contains the essential image of $y_{\Th(M)}$, so the image of $G_{\Th(M)}$ contains $im(y_{\Th(M)})$ by the commutativity of (\ref{counit}). Thus $G_{\Th(M)}^{L}|_{im(y_{M})}$ induces surjections on homotopy groups of mapping spaces. $\eta_{\Th(M)} = G_{\Th(M)}^{L}|_{im(y_{M})} \circ y_{\Th(M)}$. Thus, we conclude that $\eta_{\Th(M)}$ induces bijections on homotopy groups of mapping spaces as well. 

\end{proof}

\begin{theorem}\label{thm1.12}
$\Mt^{(-)} : \PTh_{\mathcal{A}}  \rightleftarrows \Mon_{\mathcal{E}} : \Th  $ is an idempotent adjunction, with unit $\eta$. 
\end{theorem}

\begin{proof}

By \cref{lem1.11} and \cref{lem1.8}, it remains to verify that $\epsilon, \eta$ satisfy the second of the triangle identities, i.e. that for all $\mathcal{A}$-pretheory $\mathcal{K}$, the morphism of monads $\epsilon_{\Mt^\mathcal{K}} \circ \Mt^{\eta_\mathcal{K}}$ is an equivalence. As these are morphisms of monads, we will work through the equivalence of \cref{th:monads=monadic} and instead show the induced functor between $\infty$-categories of algebras is an equivalence.

We have a commutative diagram, functorial in $\mathcal{K}$
\[
\xymatrix
{
\mathcal{E}^{\Mt^{K}} \ar[d]_{\epsilon'_{\Th(M)}} \ar[drr]^{Y_{\Mt^{\mathcal{K}}}} && \\
\mathcal{E}^{\Mt^{\Th(\Mt^{\mathcal{K}})}} \ar[rr]_{G_{\Th(\Mt^{\mathcal{K}})}} \ar[d]_{(\mathcal{E}^{\mu^{\eta_{\mathcal{K}}}})} && \Pr(\Th(\Mt^{\mathcal{K}})) \ar[d]^{\Pr(\eta_{\mathcal{K}})} \\
\mathcal{E}^{\Mt^{K}} \ar[rr]_{G_{\mathcal{K}}} && \Pr(\mathcal{K})
}
\]
where $Y_{\Mt^{\mathcal{K}}}$ is the restricted Yoneda embedding. We want to show that the composite of two left vertical functors is an equivalence.
The composite of the functor $\Pr(\eta_{\mathcal{K}}) \circ Y_{\Mt^{\mathcal{K}}}$ is given by $$x \mapsto (y \mapsto \Map_{\Pr(\mathcal{K})}(G_{\mathcal{K}}^{L} \circ y_{\mathcal{K}}(y), x)),$$ where $y_{\mathcal{K}}$ is the Yoneda embedding. This is naturally equivalent to the functor 
$$\mathcal{E}^{\Mt^{K}} \rightarrow \Pr(\mathcal{K}), x \mapsto (y \mapsto \Map_{\Pr(\mathcal{K})}(y_{\mathcal{K}}(y), G_{\mathcal{K}}(x)))$$

which is equivalent to $G_{\mathcal{K}}$, by the $\infty$-categorical Yoneda Lemma (see \cite[Proposition 5.5.2.1]{lurieHTT}, or rather \cite[Theorem 5.8.13.(ii)]{Cisinski-HigherCats} as we need the equivalence to be functorial).

Thus, we have that $G_{\mathcal{K}} \circ (\mathcal{E}^{\eta_{\mathcal{K}}})^{op} \circ \epsilon_{\Th(M)}^{op} \simeq G_{\mathcal{K}}$. Since $G_{\mathcal{K}}$ is fully faithful, and thus an equivalence onto its essential image, we conclude that $(\mathcal{E}^{\eta_{\mathcal{K}}})^{op} \circ \epsilon_{\Th(M)}^{op}$ is an equivalence by 2 out of 3. 

\end{proof}

\begin{remark} Note that there is nothing asymmetric between $\eta$ and $\epsilon$ and we have also proved that $\epsilon$ is a counit of adjunction. We just have not showed any coherence conditions between this counit $\epsilon$ and the unit $\eta$. \end{remark}

\begin{definition} A monad $M$ on $\mathcal{E}$ is said to be \emph{$\mathcal{A}$-nervous} if $\epsilon_M$ is an equivalence, i.e. if the square

\[
\begin{tikzcd}
  \mathcal{E}^M \ar[r] \ar[d] & \Prsh(\Th(M)) \ar[d] \\
\mathcal{E} \ar[r] & \Prsh(\mathcal{A})
\end{tikzcd}
\]

is a pullback square. An $\mathcal{A}$-pretheory $\mathcal{K}$ is said to be an \emph{$\mathcal{A}$-theory} if $\eta_\mathcal{K}$ is an equivalence. 
\end{definition}

The following then immediately follows from \cref{thm1.12} and Remark \ref{idempobasics}:

\begin{corollary} For any monad $M$, $\Th(M)$ is an $\mathcal{A}$-theory, and for any $\mathcal{A}$-pretheory $\mathcal{K}$, the associated monad $\mu^\mathcal{K}$ is $\mathcal{A}$-nervous. Moreover, the monad-theory adjunction restricts to an equivalence between the full subcategories of $\mathcal{A}$-Nervous monads and $\mathcal{A}$-theories. \end{corollary}

\section{General consequences of the Monad-Theories adjunction}
\label{sec:general_consequence}

In this section we draw general consequences from the monad-theory adjunction of \cref{thm1.12}. First, one can use it to construct and study colimits of $\mathcal{A}$-Nervous monads:

\begin{theorem}\label{th:colim_of_} Let $\mathcal{E}$ be a presentable $\infty$-category, and let $\mathcal{A} \subset \mathcal{E}$ be a full dense small subcategory. Then the full subcategory of $\Mon_\mathcal{E}$ of $\mathcal{A}$-Nervous monads has all colimits and they are preserved by the inclusion in $\Mon_\mathcal{E}$.  Moreover, the contravariant functor sending a monad to its category of algebras preserves these colimits. That is, the natural map:

\[ \mathcal{E}^{\colim M_i} \to \lim_{i\in I} \mathcal{E}^{M_i} \]
is an equivalence.

\end{theorem}

\begin{proof}
  The $\infty$-category of $\mathcal{A}$-pretheories is just the full subcategory of $(\icat)_{\mathcal{A}/}$ of essentially surjective functors, so it has all colimits and they are computed in $(\icat)_{\mathcal{A}/}$. This can be used to compute colimits of $\mathcal{A}$-nervous monads. Indeed, if $(M_i)_{i\in I}$ is a diagram of $\mathcal{A}$-nervous monads, then it induces a diagram $(T_i)_{i\in I}$ of $\mathcal{A}$-theories. The colimit $\colim T_i$ in the $\infty$-category of $\mathcal{A}$-pretheories exists, is preserved by the left adjoint of the monad-theory correspondence and is thus taken by this left adjoint to a colimit of the diagram $(M_i)_{i\in I}$.
  
The claim about categories of algebras actually holds for general colimits of monads (when they exist) as one can show that every object admits an endomorphism monad and one can use the universal property of the colimits for maps to endomorphism monads. Alternatively, one can also use the description of colimits given above: given that the associated monad functor sends each theory $T$ to a monad $\Mt^T$ such that $T$-models get identified functorially with $\Mt^T$-algebras, it is enough to check that the (contravariant) functor sending each pretheory to its category of models send colimits to limits. But this follows immediately from the fact that $\mathcal{C} \mapsto \Prsh(\mathcal{C}) \simeq \Fun(\mathcal{C}^{op},\mathcal{S})$ send colimits to limits.
\end{proof}

To make this useful, one needs to provide a large supply of nervous monads. The next step is \ref{th:arities_imply_nervous} that essentially claims that all accessible monads are nervous monads.

Following \cite{berger2012monads}, one defines:

\begin{definition}\label{def:monad_with_arities} Let $\mathcal{A} \subset \mathcal{E}$ be a full subcategory. Let $M$ be a monad on $\mathcal{E}$. One says that $M$ is a \textit{monad with arities in} $\mathcal{A}$ if for each $X \in \mathcal{E}$, the canonical colimit

\[ X \simeq \colim_{a \in \mathcal{A}_{/X}} a \]
is preserved the composite
\[ \mathcal{E} \overset{M}{\to} \mathcal{E} \overset{i}{\to} \Prsh(\mathcal{A}), \]
where $i$ denotes the (fully faithful) restricted Yoneda embeddings.

\end{definition}

As in in the $1$-categorical case, we will show that all monads with arities in $\mathcal{A}$ are in fact $\mathcal{A}$-nervous. The proof follows essentially the same strategy as in \cite{berger2012monads}. Note that the converse is not true, it is shown in \cite{bourke2019monads} that the free groupoid monad on the category of graphs is an example of a $\mathcal{A}$-nervous monad which is not a monad with arities in $\mathcal{A}$, for $\mathcal{A}$ the full subcategory of linear graphs.

\begin{theorem}\label{thm5.1}
Suppose that we have a commutative square of $\infty$-categories
\[
\xymatrix
{
U \ar[d]_{R_{1}} \ar[r]_{\Phi} & V \ar[d]^{R_{2}} \\
A  \ar[r]_{\Psi} & B
}
\]
where:

\begin{itemize}
\item $\Psi$ is fully faithful,
\item $R_{1}, R_{2}$ are monadic right adjoint functors, with left adjoint $L_1$ and $L_2$,
\item  the natural transformation $L_{2}\Psi \to \Phi L_{1}$ obtains from these adjunction is invertible.
\end{itemize}

Then the square is a pullback of $\infty$-categories.
\end{theorem}

\begin{proof}
We form the pullback:

\[
\begin{tikzcd}
U \ar[r,dotted,"t"description] \ar[dr,"R_1"below] \ar[rr,bend left=30,"\Phi"] &  W \ar[dr,phantom,"\lrcorner"very near start] \ar[r,"\Psi'"description] \ar[d,"R'_2"swap] & V \ar[d,"R_2"] \\
& A \ar[r,"\Psi"below] & B
\end{tikzcd}
\]

We will show that $t$ is an equivalence using \cref{cor:equiv_of_monaidc}. That is we will show that $R'_2$ is a monadic right adjoint functor and that the natural transformation $L'_2 \to t L_1$ is an equivalence of categories.

$\Psi$, and hence its pullback $\Psi'$ are both fully faithful, so up to equivalences of categories, one can freely assume that $W$ and $A$ are full subcategories of $V$ and $B$. In this case, $R'_2$ is just the restriction of $R_2$ to a functor $W \to A$. The isomorphisms $L_{2}\Psi \simeq \Phi L_{1}$ show that if $X \in A$ then $L_2 X \in W$, which  immediately implies that $L_2$ corestricted to a functor $A \to W$ is a left adjoint to $R'_2$. Hence, by \cref{prop:Pullback_monadic}, $R'_2$ is indeed a monadic functor. Now, again as we are simply restricting to full subcategories, the natural transformation $L'_2 \to t L_1$ is exactly the same as $L_2 \Psi \to \Phi L_1$ and hence is invertible. 

\end{proof}

\begin{theorem}\label{th:arities_imply_nervous} Given $\mathcal{E}$ a presentable $\infty$-category and $\mathcal{A} \subset \mathcal{E}$ a full dense small subcategory, then any monad $M$ with arities in $\mathcal{A}$ is $\mathcal{A}$-nervous.\end{theorem}

\begin{proof} For any monad $M \in \Mon_\mathcal{E}$ we have a commutative square of $\infty$-categories:

\[
\begin{tikzcd}
 \mathcal{E}^M \ar[d] \ar[r] & \Prsh(\Th_\mathcal{A}(M)) \ar[d] \\
 \mathcal{E}   \ar[r] & \Prsh(\mathcal{A}) 
\end{tikzcd}
\]

and $M$ is $\mathcal{A}$-nervous if and only if this square is a pullback. We conclude by applying \cref{thm5.1} to it. Both vertical functors are monadic right adjoint functors (for the right one, it was observed in the proof of \cref{Prop:Monadic}). The functor $\mathcal{E} \to \Prsh(\mathcal{A})$ is the restricted Yoneda embeddings and is fully faithful because $\mathcal{A}$ is dense in $\mathcal{E}$.  On the left hand side the left adjoint is the free algebra functor, and the right hand side it is the left Kan extention of the canonical functor $\mathcal{A} \to \Th_\mathcal{A}(M)$. The natural transformation ``$L_2 \Psi \to \Phi L_1$ '' in the notation of \cref{thm5.1} corresponds exactly to the map 

\[ \colim_{\mathcal{A}/X} M(a) \to M(X) \]

where the colimit is taken in $\Prsh(\Th_\mathcal{A}(M))$. This map is an equivalence if and only if its image in $\Prsh(\mathcal{A})$ is an equivalence and this corresponds exactly to the definition of a monad with arities in $\mathcal{A}$.
\end{proof}

\begin{definition}\label{def5.2}
Let $\lambda$ be a regular cardinal. We say that a monad on a $\lambda$-accessible $\infty$-category $C$ is $\lambda$\textit{-accessible} if its underlying functor is $\lambda$-accessible in the sense of \cite[5.4.2.5]{lurieHTT}. That is, if it preserves $\lambda$-directed colimits.
\end{definition}

\begin{lemma}\label{lem5.3} Let $T$ be a monad on an $\infty$-category $\mathcal{C}$ whose underlying functor commutes to colimits of $I$-shaped diagrams. Let $(C_i)_{i\in I}$ be an $I$-shaped diagram in $\mathcal{C}^T$, then:

\begin{itemize} 

\item A cocone for $C_i$ in $\mathcal{C}^T$ is a colimit cocone if and only if its image under the forgetful functor is a colimit cocone in $\mathcal{C}$.

\item If the image under the forgetful functor of $(C_i)$ admits a colimit in $\mathcal{C}$, then the colimit diagram can be lifted into a colimit diagram in $\mathcal{C}^T$.

\end{itemize} 

\end{lemma}

\begin{proof}
Let $\End_I(\mathcal{C}) \subset \End(\mathcal{C})$ be the full subcategory of endofunctors preserving $I$-shaped colimits. As $\End_I(\mathcal{C})$ is stable under composition it is a monoidal subcategory of $End(\mathcal{C})$ in the sense of section 2.2.1 of \cite{lurieHA}, and hence it is itself a monoidal $\infty$-category. A monad preserving $I$-shaped colimits can be seen as a monoid object for this subcategory. As $\mathcal{C}$ is also tensored over $End_I(\mathcal{C})$, applying \cite[Corollary 4.2.3.5]{lurieHA} to $\mathcal{C}=End_I(\mathcal{C})$ immediately gives the result claimed.\end{proof}

\begin{theorem}\label{th:nervous_monad} Let $\mathcal{E}$ be a $\lambda$-presentable category and let $\mathcal{A}$ be the full subcategory of $\lambda$-presentable objects. Then for a monad $M \in \Mon_\mathcal{E}$ the following conditions are equivalent:

\begin{enumerate}
\item\label{th:nervous_monad:accessible} $M$ is $\lambda$-accessible.
\item\label{th:nervous_monad:arities} $M$ has arities in $\mathcal{A}$.
\item\label{th:nervous_monad:nervous} $M$ is $\mathcal{A}$-nervous.
\end{enumerate}

\end{theorem}

\begin{proof}
$\ref{th:nervous_monad:accessible} \Rightarrow \ref{th:nervous_monad:arities}$ :  If $M$ is $\lambda$-accessible then $M$ preserves all $\lambda$-directed colimits. Because all objects in $\mathcal{A}$ are $\lambda$-compact, the restricted Yoneda embedding $\mathcal{E} \to \Prsh(\mathcal{A})$ preserves $\lambda$-directed colimits. Since for each $X \in \mathcal{E}$ the category $X_{/\mathcal{A}}$ is $\lambda$-directed (it has $\lambda$-small colimits) this concludes the proof.

$\ref{th:nervous_monad:arities} \Rightarrow \ref{th:nervous_monad:nervous}$ is \cref{th:arities_imply_nervous}.

$\ref{th:nervous_monad:nervous} \Rightarrow \ref{th:nervous_monad:accessible}$ : $M$ being $\mathcal{A}$-nervous means that the square:

\[
\begin{tikzcd}
  \mathcal{E}^M \ar[r] \ar[d] & \Prsh(\Th_\mathcal{A}(M)) \ar[d] \\
\mathcal{E} \ar[r] & \Prsh(\mathcal{A}) 
\end{tikzcd}
\]

is a pullback square. Now the right vertical functor preserves all colimits (in particular, $\lambda$-directed ones), and the bottom horizontal functor preserves $\lambda$-directed colimits as mentioned above. It hence follows that all functors in the diagram preserve $\lambda$-directed colimits by \ref{colimitinternal}. The underlying functor of the monad $M$ identifies with the composite of the forgetful functor $\mathcal{E}^M \to \mathcal{E}$ and its left adjoint (which automatically preserves colimits), so it preserves $\lambda$-directed colimits. Thus, $M$ is $\lambda$-accessible.
\end{proof}






\begin{corollary}\label{cor5.5}
Let $M$ be a $\lambda$-accessible monad on a $\lambda$-presentable $\infty$-category $\mathcal{E}$. Then the  $\infty$-category $\mathcal{E}^M$ of $M$-algebra is locally presentable. In particular it has all colimits. 
\end{corollary}

\begin{proof}
With $\mathcal{A} $ the full subcategory of $\lambda$-presentable objects, we have by \cref{th:nervous_monad} pullback diagram:
\[
\begin{tikzcd}
  \mathcal{E}^M \ar[r] \ar[d] & \Prsh(\Th_\mathcal{A}(M)) \ar[d] \\
\mathcal{E} \ar[r] & \Prsh(\mathcal{A}) 
\end{tikzcd}
\]

$\Prsh(M^{\lambda}), \Prsh(C^{\lambda})$ are locally presentable by \cite[Theorem 5.5.1.1]{lurieHTT}. The vertical right map preserve all limits and all colimits so it is an accessible right adjoint functor and the bottom horizontal map preserves all limits and $\lambda$-directed colimits, so it is also an accessible right adjoint. It then follows from \cite[Theorem 5.5.3.18]{lurieHTT} that taking this pullback in the category of presentable categories and right adjoint functors between them gives the same results, and hence $\mathcal{E}^M$ is itself locally presentable.
\end{proof}

\begin{corollary}\label{cor:colimits_of_monads} Let $\mathcal{E}$ be a locally presentable category and $M:I \to \Mon_\mathcal{E}$ a diagram such that $M(i)$ is accessible for each $i \in I$, then $M$ has a colimit in $\Mon_\mathcal{E}$ and the natural map:

\[\mathcal{E}^{\colim M_i} \to \lim_{i \in I} \mathcal{E}^{M_i} \]

is an equivalence of $\infty$-categories.

\end{corollary}

More precisely, the proof will show that if $\mathcal{E}$ is $\kappa$-presentable and all $M(i)$ are $\kappa$-accessible then the colimit is $\kappa$-accessible.

\begin{proof} Given $\kappa$ a regular cardinal such that $\mathcal{E}$ is $\kappa$-presentable and all $M(i)$ are $\kappa$-accessible, \cref{th:nervous_monad} shows that all $M(i)$ are $\mathcal{A}$-nervous for $\mathcal{A}$ the category of $\kappa$-compact objects in $\mathcal{A}$, and \cref{thm5.1} implies the result. \end{proof}

\section{Monads as Kleisli categories}
\label{sec:Monads_as_Kleisli}

The goal of this section is to show that one can works with a monad purely in terms of its Kleisli category, so that defining a monad on $\mathcal{C}$ is the same as defining a bijective on objects left adjoint functor $\mathcal{C} \to \mathcal{K}$. This section is generally independant of the rest of the paper, but uses very similar methods and fits in the general goal of providing tools to work more easily with monads on $\infty$-categories.

\begin{definition} Let $\Ladj_{\mathcal{C}} $ be the full subcategory of $(\icat)_{\mathcal{C}/}$ on \emph{left adjoint essentially surjective functors}. \end{definition}

Let $\Kl: \Mon_{\mathcal{C}} \rightarrow \Ladj_{\mathcal{C}}$ be the Kleisli category construction. The main result of this section is:

\begin{theorem}\label{th:Mon_as_Kleisli_main}
The functor $\Kl$ is an equivalence of $\infty$-categories between the $\infty$-categories $\Mon_\mathcal{C}$ and $\Ladj_\mathcal{C}$.
\end{theorem}

As well, the following proposition allows us to recover the $\infty$-category of algebras of a monad out of its Kleisli categories.

\begin{proposition}\label{prop:EM_pb_of_Kl}
Let $\mathcal{C}^{M} \rightarrow \mathcal{C}$ be a monadic functor
The square
\[
\xymatrix
{
\mathcal{C}^{M} \ar[r] \ar[d] & \Pr(\mathcal{C}_{M})  \ar[d] \\
\mathcal{C} \ar[r] & \Pr(\mathcal{C})
 }
\]
where the horizontal arrows are the restricted Yoneda embeddings is a pullback.
\end{proposition}

\begin{proof}
In the diagram, the vertical maps are monadic, and the bottom horizontal map is fully faithful. 
By \ref{thm5.1}, we must show that the adjoint natural transformation (``$L_2 \Psi \to \Phi L_1$ '' in the notation of \ref{thm5.1}) is an equivalence. But this was done within the proof of \cref{yonedanat}, when checking that \cref{prop:partial_adj_natural} can be applied. 
\end{proof}

A key observation is that the pullback of \cref{prop:EM_pb_of_Kl} allows us to associate a monad on $\mathcal{C}$ to every essentially surjective left adjoint functor $L:\mathcal{C} \to \mathcal{K}$.

\begin{lemma}\label{lem:Kl_pb_monadicity} Let $F: \mathcal{C} \to \mathcal{K}$ be an essentially surjective left adjoint functor, then, in the pullback square:

\[ \begin{tikzcd}
\mathcal{M} \ar[r] \ar[d] \ar[dr,phantom,"\lrcorner"very near start] & \Pr(\mathcal{K}) \ar[d] \\
\mathcal{C} \ar[r] & \Pr(\mathcal{C})
\end{tikzcd} \]

The functor $\mathcal{M} \to \mathcal{C}$ is a monadic right adjoint.
\end{lemma}

\begin{proof} The proof is the same as in \cref{Prop:Monadic} except for the part about the existence of a left adjoint functor $\mathcal{C} \to \mathcal{M}$ (which in \cref{Prop:Monadic} follows from a presentability argument). Because $F:\mathcal{C} \to \mathcal{K}$ has a right adjoint $R$, the restriction functor $F^* : \Pr(\mathcal{K}) \to \Pr(\mathcal{C})$ sends the representable at $X \in \mathcal{K}$ to the representable at $R(X) \in \mathcal{C}$, and (as for any functor $F$), its left adjoint functor $F_!: \Pr(\mathcal{C}) \to \Pr(\mathcal{K})$ sends representables to representables. It follows that, as $\mathcal{C}$ and $\mathcal{M}$ are respectively full subcategories of $\Pr(\mathcal{C})$ and $\Pr(\mathcal{K})$ preserved by the action of $F^*$ and $F_!$, the restriction of $F_!$ to a functor $\mathcal{C} \to \mathcal{M}$ is a left adjoint to the restriction of $F^*:\mathcal{M} \to \mathcal{C}$.
\end{proof}

\begin{construction}\label{cstr:Omega_functor} \Cref{lem:Kl_pb_monadicity} allows us to construct a functor $\Omega: \Ladj \to \Mon_\mathcal{C}$, or more precisely, a functor $\Ladj^{op} \to \Md_\mathcal{C}$. The construction that sends an essentially surjective left adjoint functor $F:\mathcal{C} \to \mathcal{K}$ to the pullback $\mathcal{M} \to \mathcal{C}$ as in \cref{lem:Kl_pb_monadicity} is a contravariant functor: The presheaf construction (with its contravariant functorialiry) defines a functor $((\icat)_{\backslash \mathcal{C}})^{op} \to (\icat)_{/\Pr(\mathcal{C})}$ (up to some easily dealt with size issues) which can be composed with the pullback functor $(\icat)_{/\Pr(\mathcal{C})} \to (\icat)_{/\mathcal{C}}$. Finally \cref{lem:Kl_pb_monadicity} shows that this functors sends the full subcategory $\Ladj_\mathcal{C}$ to $\Md_\mathcal{C}$.

\end{construction}

We conclude the proof of \cref{th:Mon_as_Kleisli_main}, with:

\begin{proposition} The functor $\Omega: \Ladj_{\mathcal{C}} \to \Mon_\mathcal{C}$ of \cref{cstr:Omega_functor} is an inverse for $\Kl: \Mon_\mathcal{C} \to \Ladj$.\end{proposition}

\begin{proof} We will construct two explicit natural isomorphisms $ \Omega \circ \Kl(M) \to M $ and $\Kl \circ \Omega(\mathcal{K}) \to \mathcal{K}$.

By \cref{cor:rest_yon_fct} the restricted Yoneda embedding $\mathcal{C}^M \to \Pr(\mathcal{C}_M)$ is natural in $M$. Given the pullback defining the category of algebras of $\Omega(\mathcal{C}_M)$ this translated into a map, natural in $M$, from $\mathcal{C}^M$ to that category of algebras, which by \cref{prop:EM_pb_of_Kl} is an equivalence. Though the equivalence of \cref{th:monads=monadic}, this translate to a isomorphism of monads $M \to \Omega \circ \Kl (M)$.

Given $F:\mathcal{C} \to \mathcal{K}$ in $\Ladj$, recall that the category of algebras $\mathcal{C}^{\Omega(F)}$ is constructed (functorially) as the pullback:

\[ \begin{tikzcd}
\mathcal{C}^{\Omega(F)} \ar[d] \ar[r] & \Pr \mathcal{K} \ar[d,"F^*"] \\
\mathcal{C} \ar[r] & \Pr \mathcal{C}
\end{tikzcd}
\]

Its Kleisli category is the essentially image of the left adjoint of $\mathcal{C}^{\Omega(F)} \to \mathcal{C}$ and it is made functorial by \cref{prop:Partial_adjoint_fct}. It hence follows from \cref{prop:partial_adj_natural} (that the assumption are satisfied follows from the proof of \cref{lem:Kl_pb_monadicity}) that we have a natural transformation $\mathcal{C}_{\Omega(F)} \to \Pr \mathcal{K}$ where $\Pr$ has its covariant/left adjoint functoriality\footnote{We refer again to section 6 of \cite{hebestreit2020orthofibrations} for the fact that the two possible definition of this covariant functoriality are equivalent.}. Now the explicit construction of the left adjoint to $\mathcal{C}^{\Omega(F)} \to \mathcal{C}$ done in the proof of \cref{lem:Kl_pb_monadicity} shows that the functor  $\mathcal{C}_{\Omega(F)} \to \Pr \mathcal{K}$ induces an equivalence between $\mathcal{C}_{\Omega(F)}$ and the full subcategory of $\Pr \mathcal{K}$ of representable presheaves (which is essentially $\mathcal{K}$). As the Yoneda embedding of $\mathcal{K}$ into $\Pr \mathcal{K}$ is natural for this left adjoint/covariant functoriality of $\Pr$ (again by section 6 of \cite{hebestreit2020orthofibrations}), this boils down to a natural equivalence (under $\mathcal{C}$) $\mathcal{C}_{\Omega(F)} \simeq \mathcal{K}$  which concludes the proof.
\end{proof}

\section{$E_{1}$, $E_2$ and $E_\infty$-algebras}
\label{sec:E1_algebra}

In this section we show that the monads on the $\infty$-category $\Sp$ of spaces corresponding to the $E_1$, $E_2$ and $E_\infty$-operads can be seen respectively as "induced" the free monoid monad on Set, the free braided monoid on groupoids and the free symmetric monoid on groupoids. By induced here we mean that when restricted to appropriate category of arities they corresponds to the same theories. 

It should be noted that the $E_2$ and $E_\infty$ operads cannot be described by the framework of planar operads that we recalled in \cref{sec:monads_prelim}. It needs the more general ``symmetric'' operads framework. We will not recall the details of this and we refer directly to \cite{lurieHA}. However, to fix notation, we note that, similarly to how a planar operad is encoded by a map $\mathcal{O}^\circledast \to N(\Delta^{op})$, a symmetric operad is encoded by a map $\mathcal{O}^{\otimes} \rightarrow N(\mathrm{Fin}_{*})$ of $\infty$-categories, where $\mathrm{Fin}_{*}$ is the category of finite pointed sets. 

We first recall some basic facts about sifted diagrams:

\begin{definition}\label{def6.1}
An $\infty$-category $K$ is said to be \emph{sifted} if the diagonal map $K \rightarrow K \times K$ is cofinal.
\end{definition}

\begin{remark}\label{rmk6.2}
Note that the property of being sifted is invariant under equivalence of $\infty$-categories (see \cite[Corollary 4.1.1.10]{lurieHTT}).
\end{remark}

\begin{lemma}\label{lem6.3}
Suppose that $K$ is an $\infty$-category that has finite coproducts. Then $K$ is sifted. 
\end{lemma}

\begin{proof}
 By \cite[4.1.3.1]{lurieHA}, it suffices to show that for all $a, b \in K$, $K \times_{K \times K} (K \times K)_{(a, b)/} \cong  K_{b/} \times_{K} K_{a/} \cong K_{\{ a, b\}/}$ is weakly contractible. But this $\infty$-category is weakly contractible since it has an initial object, the coproduct of $a, b$. 
\end{proof}


We say that an $\infty$-operad $\mathcal{O}^{\otimes}$ is a \emph{non-colored $\infty$-operad} if its underlying $\infty$-category is terminal, i.e. if $\mathcal{O} \cong \Delta^{0}$ (see \cite[Example 2.1.1.6]{lurieHA}). When $\mathcal{O}$ is a non-colored $\infty$-operad, we have a forgetful functor $\mathrm{Alg}_{\mathcal{O}^{\otimes}}(\mathcal{B}) \rightarrow \mathcal{B}$ for any symmetric monoidal $\infty$-category $\mathcal{B}$ (or more generally any $\mathcal{O}$-monoidal $\infty$-category).

The goal of the next few paragraphs is to show that given a non-colored $\infty$-operad $\mathcal{O}^{\otimes}$, then the forgetful functor $\mathrm{Alg}_{\mathcal{O}^{\otimes}}(\mathcal{B}) \rightarrow \mathcal{B}$ where $\mathcal{B}$ is one of the (cartesian) symmetric monoidal $\infty$-categories $\mathrm{Set}, \mathrm{Gdp}$ or $\mathcal{S}$, is monadic and the associated monad is $\mathrm{Fin}$-nervous, where $\mathrm{Fin} \subset \mathcal{B}$ is the full subcategory of finite sets.

Recall that the $\infty$-category of spaces $\mathcal{S}$, as well as its full subcategory $\mathrm{Set}$ and $\mathrm{Gpd}$ of sets (i.e. discrete spaces) and groupoids (i.e. $1$-truncated spaces), are cartesian closed locally presentable $\infty$-categories. In particular \cref{lem6.4} and \cref{lem6.5} below can be applied to them.

\begin{lemma}\label{lem6.4}
Let $\mathcal{O}^{\otimes}$ be a non-colored $\infty$-operad and $\mathcal{C}$ a locally presentable cartesian closed symmetric monoidal $\infty$-category.

Then $\infty$-category $Alg_{\mathcal{O}^{\otimes}}(\mathcal{S})$ has all sifted colimits and the forgetful functor $Alg_{\mathcal{O}^{\otimes}}(\mathcal{C}) \rightarrow  \Fun(\mathcal{O}, \mathcal{C}) \simeq \mathcal{C}$ preserves sifted colimits. 
\end{lemma}

\begin{proof}
 For the first statement \cite[Proposition 3.2.3.1]{lurieHA} implies that it suffices to show that for $n \in \mathbb{N}$, the induced map $\mathcal{C}^{\otimes}_{[n]} \rightarrow \mathcal{C}^{\otimes}_{[1]}$ (see \cite[Remark 2.1.2.6]{lurieHA}), preserves sifted colimits separately in each variable. Because $\mathcal{C}$ is cartesian, this functor can be identified with the functor $\mathcal{C}^{n} \rightarrow \mathcal{C}$ that takes a collection of objects to their n-fold product. But since $\mathcal{C}$, is cartesian closed, products preserve sifted colimits seperately in each variable, hence the result.
 
The fact that the forgetful functor preserves all sifted colimits follows from another application of \cite[Proposition 3.2.3.1]{lurieHA}.

\end{proof}

The left adjoint of the forgetful functor $Alg_{\mathcal{O}^{\otimes}}(\mathcal{C}) \rightarrow \mathcal{C}$ (if it exists) is called the \emph{free $\mathcal{O}$-algebra functor} and is denoted $\text{Free}_{\mathcal{O}}^{\mathcal{C}}$.

\begin{lemma}\label{lem6.5}
Let $\mathcal{O}^{\otimes}$ and $\mathcal{C}$ as in \cref{lem6.4}. Then the forgetful functor $Alg_{\mathcal{O}^{\otimes}}(\mathcal{C}) \rightarrow \mathcal{C}$ is a monadic right adjoint functor. 
\end{lemma}

\begin{proof}
We verify the three hypotheses of Barr-Beck-Lurie. Since colimits in $\mathcal{C}$ are preserved by the products and $\mathcal{C}$ is presentable, it follows from \cite[Example 3.1.3.6]{lurieHA} and \cref{lem6.4} that the functor is a right adjoint. Since $N(\Delta^{op})$ is sifted (\cite[Lemma 5.5.8.3]{lurieHTT}), \ref{lem6.4} implies that it preserves colimits of split simplicial objects. The functor reflects limits (\cite[Corollary 3.2.2.5]{lurieHA}) and hence reflects equivalences; the limit of a diagram $X: \Delta^{0} \rightarrow \mathcal{C}$ is just an object equivalent to $X$. 
\end{proof}





\begin{lemma}\label{lem6.7}
For each $s \in \mathcal{S}$, the category $\mathrm{Fin}_{/s}$ is sifted.
\end{lemma}

\begin{proof}
Coproducts in $\mathcal{S}_{/s}$ are computed as coproducts in $\mathcal{S}$, in particular $\mathrm{Fin}_{/s}$, seen as a full subcategory of $\mathcal{S}_{/s}$ is closed under finite coproducts because $\mathrm{Fin}$ is closed under finite coproducts in $\mathcal{S}$. The result then follows from \cref{lem6.3}.
\end{proof}

\begin{theorem}\label{thm6.8}
Suppose that $\mathcal{B} = \mathcal{S}, \mathrm{Gdp}$ or $\mathrm{Set}$. Let $\mathcal{O}^\otimes$ be a non-colored $\infty$-operad.
Then the monad on $\mathcal{B}$ corresponding the forgetful functor $Alg_{\mathcal{O}^{\otimes}}(\mathcal{B}) \rightarrow \mathcal{B}$ is $\mathrm{Fin}$-nervous. 
\end{theorem}

\begin{proof}

We will show more precisely that this monad, which we denote $M$, has arities in $\mathrm{Fin}$, in the sense of \cref{def:monad_with_arities}, which implies the result by \ref{th:arities_imply_nervous}. It suffices to show that the functor 
\[
\mathcal{B} \xrightarrow{M} \mathcal{B} \xrightarrow{i} \mathrm{Pr}(\mathrm{Fin})
\]
preserves $\mathrm{colim}_{a \in \mathrm{Fin}/X}(a)$ for each $X \in \mathcal{B}$.  By \ref{lem6.7}, it suffices to show that $M$ and $i$ preserve sifted colimits. The monad $M$ is the composite of the left adjoint $\text{Free}_{\mathcal{O}}^{\mathcal{B}}$, which preserves all colimits, and the forgetful functor $Alg_{\mathcal{O}^{\otimes}}(\mathcal{B}) \rightarrow \mathcal{B}$ which preserves sifted colimits by \cref{lem6.4}. Hence $M$ preserves sifted colimits.
 
 It suffices to show that the restricted Yoneda embedding $i$ preserves sifted colimits. Since colimits in $\mathrm{Pr}(\mathcal{A})$ are calculated pointwise, it suffices to show that for each $K \in \mathrm{Fin}$ and sifted $\infty$-category I, the natural map
 \[
 \mathrm{colim}_{i \in I} \Map_{\mathcal{S}}(K, a_{i}) \rightarrow \Map_{\mathcal{S}}(K, \mathrm{colim}_{i \in I} a_{i}) 
 \]
 is an equivalence. This can be identified with the map
 \[
 \prod_{j \in K} (\mathrm{colim}_{i \in I} a_{i}) \rightarrow \mathrm{colim}_{i \in I} \prod_{j \in K} a_{i}
 \]
 In other words, we want to show that sifted colimits preserve finite products in $\mathcal{B}$, which follows from $\mathcal{B}$ being cartesian closed and \cite[Proposition 5.5.8.6 and Lemma 5.5.8.11]{lurieHTT}.

\end{proof}

\begin{lemma}\label{lem6.9}
Suppose that  $G: \mathcal{C} \rightarrow \mathcal{D}$ is a fully faithful functor of $\infty$-categories, and $\mathcal{E}$ be an $\infty$-category. Then $\Fun(\mathcal{E}, \mathcal{C}) \rightarrow \Fun(\mathcal{E}, \mathcal{D})$ is fully faithful. 
\end{lemma}

\begin{proof} Up to equivalence of $\infty$-categories, one can assume that $\mathcal{C}$ is a full subcategory of $\mathcal{D}$, in which case $\Fun(\mathcal{E}, \mathcal{C})$ is isomorphic (as a simplicial) set to the full subcategory of $\Fun(\mathcal{E},\mathcal{D})$ of functors that sends all objects of $\mathcal{E}$ to $\mathcal{D}$.

\end{proof}


Suppose that $\mathcal{B} \subseteq \mathcal{S}$ is either $\mathrm{Set}$, $\mathrm{Gpd}$. We write $\Mt_{\mathcal{B}}^{(-)} \dashv \Th_{\mathcal{B}}$ for the adjunction of \ref{thm1.12} coming from the inclusion of arities $\mathrm{Fin} \subseteq \mathcal{B}$.

\begin{theorem}\label{thm6.10}
Let $\mathcal{B} \subsetneq \mathcal{S}$ be as above. Let $\mathcal{O}^{\otimes}$ be a non-colored $\infty$-operad. Suppose that the free algebra functor $\mathcal{S} \rightarrow Alg_{\mathcal{O}^{\otimes}}(\mathcal{S})$ takes elements of $\mathcal{B}$ to $Alg_{\mathcal{O}^{\otimes}}(\mathcal{B})$. Then there exists a theory $(\mathrm{Fin} \rightarrow \mathcal{K}) \in \PTh_{\mathrm{Fin}}$, so that 

\[
\Sp^{\Mt^{\mathcal{K}}_{\mathcal{S}}} \simeq \Mod_{\mathcal{K}}(\mathcal{S}) \simeq Alg_{\mathcal{O}^{\otimes}}(\mathcal{S}) \qquad \mathcal{B}^{\Mt^{\mathcal{K}}_{\mathcal{B}}} \simeq \Mod_{\mathcal{K}}(\mathcal{B}) \simeq Alg_{\mathcal{O}^{\otimes}}(\mathcal{B}).
\]

Moreover, $\mathrm{Fin} \to \mathcal{K}$ is a theory with respect to both for $\mathrm{Fin} \subset \mathcal{S}$ and $\mathrm{Fin} \subset \mathcal{B}$.

\end{theorem}

\begin{remark}
Note that in particular, if $\mathcal{B}$ is a $1$-category, i.e. when $\mathcal{B} = \mathrm{Set}$, then $\mathcal{K}$ is a $1$-category. To see this, note that $\mathrm{Alg}_{\mathcal{O}^{\otimes}}(\mathcal{B})$ can be identified with a full subcategory of $\Fun(\mathcal{O}^{\otimes}, \mathcal{B})$ by \cite[Proposition 2.4.1.7]{lurieHA}, and is hence a 1-category by \cite[Corollary 2.3.4.20]{lurieHTT}. But $\mathcal{K}$ is by definition a full subcategory of $\mathrm{Alg}_{\mathcal{O}^{\otimes}}(\mathcal{B})$, so the result follows. Similarly, if $\mathcal{B}$ is a $2$-category, or rather a $(2,1)$-category, i.e. when $\mathcal{B} = \mathrm{Gpd}$, then $\mathcal{K}$ is also itself a $2$-category.
\end{remark}

\begin{proof} Let $\mathcal{S}^\otimes \to N(\mathrm{Fin}_*)$ and $\mathcal{B}^\otimes \to N(\mathrm{Fin}_*)$ be the $\infty$-operads corresponding to the cartesian monoidal structure on $\mathcal{S}$ and $\mathcal{B}$ (as explained in section 2.1.1 of \cite{lurieHA}).

Consider the diagram
\begin{equation}\label{subalg}
\xymatrix
{
Alg_{\mathcal{O}^{\otimes}}(\mathcal{B}) \ar[r] \ar[d]_{F_{1}} & \mathrm{Alg}_{\mathcal{O}^{\otimes}}(\mathcal{S}) \ar[d]_{F_{2}} \\
\mathcal{B} \ar[r] & \mathcal{S}
}
\end{equation}

First, we note that the top horizontal map is fully faithful. Indeed, the categories of $\mathcal{O}$-algebras are full subcategory of the categories of functor $\Fun_{/\mathrm{Fin}_*}(\mathcal{O}^{\otimes}, \mathcal{B}^\otimes)$ and $\Fun_{/\mathrm{Fin}_*}(\mathcal{O}^{\otimes}, \mathcal{S}^\otimes)$ over $\mathrm{Fin}_*$. But the functor $\Fun_{/\mathrm{Fin}_*}(\mathcal{O}^{\otimes}, \mathcal{B}^\otimes) \to \Fun_{/\mathrm{Fin}_*}(\mathcal{O}^{\otimes}, \mathcal{S}^\otimes)$ is fully faithful because it is a pullback of $\Fun(\mathcal{O}^{\otimes}, \mathcal{B}^\otimes) \to  \Fun(\mathcal{O}^{\otimes}, \mathcal{S}^\otimes)$ which is fully fatihfull by \ref{lem6.9}.

Let $M_{1}, M_{2}$ be the monads associated to the left and right vertical maps of \ref{subalg}, respectively.
Since the horizontal maps are fully faithful, we can without loss of generality treat the horizontal maps as inclusions of full subcategories. The restriction of the counit of $H_{2} \dashv F_{2}$ gives the counit of the adjunction $H_{2}|_{\mathcal{B}} : \mathcal{B} \leftrightarrows Alg_{\mathcal{O}^{\otimes}}(\mathcal{B}) : F_{1}$, since $H_{2}$ takes objects of $\mathcal{B}$ to  $Alg_{\mathcal{O}^{\otimes}}(\mathcal{B})$. Consider the composites
\begin{equation*}
\mathrm{Fin} \subseteq \mathcal{B}  \xrightarrow{H_{2}|_{\mathcal{B}}}  Alg_{\mathcal{O}^{\otimes}}(\mathcal{B}) \qquad  \mathrm{Fin} \subseteq \mathcal{S} \xrightarrow{H_{2}} Alg_{\mathcal{O}^{\otimes}}(\mathcal{S})
\end{equation*}
the essential images of which correspond to $\Th_{\mathcal{B}}(M_{1}), \Th_{\mathcal{S}}(M_{2})$. These composites are the same, since $\mathrm{Fin} \subseteq \mathcal{B}$. We will denote the composite by $\mathrm{Fin} \rightarrow \mathcal{K}$.

But by \ref{thm6.8}, $M_{1}, M_{2}$ are both $\mathrm{Fin}$-Nervous, so that $M_{1} \cong \Mt_{\mathcal{B}}^{\Th(M_{1})} \cong \Mt_{\mathcal{B}}^{\mathcal{K}}$, $M_{2} \cong \Mt_{\mathcal{S}}^{\Th(M_{2})} \cong \Mt_{\mathcal{S}}^{\mathcal{K}}$.

\end{proof}

\begin{remark}\label{rmk6.11}
In the situation of \ref{thm6.10} the proof implies that  $\text{Free}_{\mathcal{O}}^{\mathcal{B}}$ can be identified with $\text{Free}_{\mathcal{O}}^{\mathcal{S}}|_{\mathcal{B}}$. Thus, we can think of $\text{Free}_{\mathcal{O}}^{\mathcal{S}}$ as extending $\text{Free}_{\mathcal{O}}^{\mathcal{B}}$.
\end{remark}

\begin{example}\label{exam6.12}
Let $E_{1}^{\otimes}$ be the $E_{1}$-operad studied in \cite[Chapter 5]{lurieHA}. Using \cite[Example 5.1.0.7]{lurieHA} we can identify this with the associative operad $\mathrm{Assoc}^{\otimes}$. By \cite[Proposition 4.1.1.18]{lurieHA}, the free monad functor $  \mathcal{S} \rightarrow \mathrm{Alg}_{E_{1}^{\otimes}}(\mathcal{S})$ takes $C$ to an algebra with underlying object $\coprod_{n \in \mathbb{N}} C^{n}$. Since (co)products in the $\infty$-category of spaces can be identified with ordinary (co)products, the free algebra functor preserves the property of having the homotopy type of a set. 

Thus, we can apply \ref{thm6.10} with $\mathcal{B} = \mathrm{Set}, \mathcal{O}^{\otimes}  = E_{1}^{\otimes}$ and \ref{rmk6.11}, to conclude that the ``free-$E_{1}$-space''-monad on $\mathcal{S}$ extends the ``free monoid monad'' on sets. 

By the rectification result of \cite[Theorem 4.1.8.4]{lurieHA}, $Alg_{Assoc^{\otimes}}(\mathrm{Set}) \rightarrow \mathrm{Set}$ can be identified with the forgetful functor $\textbf{Monoid} \rightarrow \mathrm{Set}$, which takes a monoid in the classical sense to its underlying set. Thus, the 'free monoid monad' constructed above can be identified with the classical free monoid monad from \cite[Example 9]{bourke2019monads}). Moreover, if $\mathcal{K}$ is the classical algebraic theory from \cite{bourke2019monads} whose set-valued models are monoids, then its models in $\mathcal{S}$ can be identified with the $E_{1}$-spaces. 

\end{example}

\begin{lemma}\label{lem6.13}
Let $\mathrm{Comm}^{\otimes}$ be the commutative (or $E_{\infty}$) operad studied in \cite[Example 2.1.1.8]{lurieHA}.
The free algebra functor $\mathcal{S} \rightarrow \mathrm{Alg}_{\mathrm{Comm^{\otimes}}}(S)$ takes elements of $\mathrm{Gpd}$ to elements of $\mathrm{Alg}_{Comm^{\otimes}}(\mathrm{Gpd})$.
\end{lemma}
\begin{proof}
By \cite[Example 3.1.3.14]{lurieHA}, the left adjoint to the forgetful functor is given by $C \mapsto \coprod_{n \ge 0} \mathrm{Sym}^{n}(C)$, where $\mathrm{Sym}^{n}$ is as in \cite[Construction 3.1.3.9]{lurieHA}. Thus, it suffices to show that $\mathrm{Sym}^{n}(-)$ takes groupoids to groupoids. 

Let  $\Sigma_n$ be the symmetric group regarded as a category with one object. Unwinding \cite[Construction 3.1.3.9]{lurieHA}, $\mathrm{Sym}^{n}(C)$ gets identified with the colimit of a diagram $N(\Sigma_{n}) \rightarrow \mathcal{S}$ which takes the object to $C^{n}$ and acts by permuting the factors. This can be further identified with the homotopy colimit of a group acting on a space.

Such a homotopy colimit is called a \emph{homotopy orbit space}, and it fits into a homotopy fibre sequence 
$$
C^{n} \rightarrow \mathrm{Sym}^{n}(C) \rightarrow N(\Sigma_{n})
$$ 
(for a description of homotopy orbit spaces, and the above fibre sequence, see \cite[Chapter 1, Section 6]{GroupCohomology}). The long exact sequence of homotopy groups associated to the above fibre sequence shows that since $N(\Sigma_{n}), C^{n}$ are groupoids, so is $\mathrm{Sym}^{n}(C)$. 

\end{proof}

\begin{example}\label{exam6.14}
By the preceding lemma we can apply \ref{thm6.10}, \ref{rmk6.11} with $\mathcal{O}^{\otimes} = E_{\infty}^{\otimes}, \mathcal{B} = \mathrm{Gpd}$ to show that the monad $\text{Free}_{E_{\infty}}^{\mathcal{S}}$ extends $\text{Free}_{E_{\infty}}^{\text{Gpd}}$. In other words, the free symmetric monoidal groupoid monad is extended by the Free $E_{\infty}$-space monad. 

Using \cite[Example 2.4.2.5]{lurieHA} and \cite[Proposition 2.4.2.4]{lurieHA}, we see that the objects of $\mathrm{Alg}_{E_{\infty}^{\otimes}}(Gpd)$ can be identified with symmetric monoidal groupoids. By the definition of 1-morphisms in this $\infty$-category can be identified with functors $F : A \rightarrow B$ of symmetric monoidal categories, along with isomorphism $F(- \otimes_{A}- ) \cong F(-) \otimes_{B} F(-)$, compatible with the commutativity and associativity properties of $A$ and $B$. In other words, they can be identified with monoidal functors. Similarly the 2-morphisms in $\mathrm{Alg}_{E_{\infty}^{\otimes}}(Gpd)$ can be identified with monoidal natural transformations. Thus we can identify $\text{Free}_{E_{\infty}}^{\text{Gpd}}$ with the classical free symmetric monoidal groupoid monad considered in \cite{2D-Monads}.

 \end{example}

\begin{example}\label{exam6.15}
The free $E_{2}$-algebra $\mathcal{S} \rightarrow \mathrm{Alg}_{E_{2}^{\otimes}}(\mathcal{S})$ takes an object $X$ to $\coprod_{n \in \mathbb{N}} B^{n}(X)$, where $B^{n}(X)$ is the colimit of the braid group action on $X^{n}$. This functor takes $\mathrm{Gpd}$ to $\mathrm{Alg}_{E_{2}^{\otimes}}(\mathrm{Gpd})$, by the same argument as \ref{lem6.13}. As noted in \cite[Example 5.1.2.4]{lurieHA}, the objects of $\mathrm{Alg}_{E_{2}^{\otimes}}(\mathrm{Gpd})$ can be identified with braided monoidal groupoids. Thus, as in the preceding example, we can conclude that the free braided monoidal groupoid monad is extended by the free $E_{2}$-space monad. 

\end{example}

\begin{remark}\label{rmk6.16}
If $n \ge 3$, is not possible to find a monad on $\mathrm{Gpd}$ whose algebraic theory has as its $\mathcal{S}$-models $E_{n}$-spaces. The reason is that by \cite[Corollary 5.1.1.7]{lurieHA}, $E_{\infty}$-algebras and $E_{n}$ algebras in $\mathrm{Gpd}$ coincide for $n \ge 3$, so the existence of a theory with the required properties would imply that $E_{\infty}$-spaces are the same as $E_{n}$-spaces. The aforementioned fact can be viewed as an analogue of the Baez-Dolan stabilization hypothesis (see \cite{Baez-Dolan} and \cite[Example 5.1.2.3]{lurieHA}).

It should be noted that for all $2<n< \infty$ , the free $E_n$-algebra on a set $X$ has homotopy groups in arbitrary large dimension, i.e. is not $k$-truncated for any $k$. So replacing $\mathcal{B}$ by the category of $k$-groupoids for a larger $k$ does not allow one to deal with the case of $E_n$-algebra for larger $n$ even if the argument above does not obstruct it. \end{remark}

\section{Relation to algebraic patterns}
\label{sec:chu_haugseng}

Finally, we clarify the relation between our results and Chu and Haugseng's theory of algebraic patterns from \cite{chu2019homotopy}. In a very simplified way, algebraic patterns are a type of ``theory'' that through the monad-theory adjunction corresponds to cartesian parametric right adjoint\footnote{which are called polynomial monads in \cite{chu2019homotopy}.} monads on presheaf categories.

A natural transformation is said to be \emph{cartesian} if all of its naturality squares are cartesian. A monad is said to be cartesian if its unit and composition natural transformation $Id \to M$ and $ M \to M^2$ are cartesian. This also implies that all other structural morphisms of the monad are cartesian. A parametric right adjoint monad is a monad whose underlying functor $M: \mathcal{C} \to \mathcal{C}$ admits a right adjoint when considered as a functor $\mathcal{C} \to \mathcal{C}/M(1)$ for $1$ a terminal object of $\mathcal{C}$.

\bigskip

Note that \cite{chu2019homotopy} defines models in terms of covariant functors to Set while we use presheaves, i.e. contravariant functors as in the $1$-categorical tradition (like \cite{berger2012monads} or \cite{bourke2019monads}). To simplify the connection between the present paper and \cite{chu2019homotopy}, we will rephrase the definitions given in \cite{chu2019homotopy} in terms of the opposite categories.

\bigskip

A \emph{categorical pattern} in the sense of \cite{chu2019homotopy} is a category $\mathcal{O}$ endowed with a factorization system $(\mathcal{O}^{act},\mathcal{O}^{in})$ whose left class is called the \emph{active morphisms} and the right class is called the \emph{inert morphisms}, and a full subcategory $\mathcal{O}^{el} \subset \mathcal{O}^{in}$ of objects called elementary objects.

\bigskip

Given a categorical pattern $(\mathcal{O},\mathcal{O}^{act},\mathcal{O}^{in},\mathcal{O}^{el})$, a Segal $\mathcal{O}$-object is a presheaf $\mathcal{F}$ on $\mathcal{O}$ which satisfies the following equivalent conditions:

\begin{itemize}

\item For all $X \in \mathcal{O}$, the map 

\[ \mathcal{F}(X) \to \lim_{E \to X \in \mathcal{O}^{in} \atop E \in \mathcal{O}^{el} } \mathcal{F}(E) \]

is an equivalence.

\item The restriction of $\mathcal{F}$ to $\mathcal{O}^{in}$ is a right Kan extension of $\mathcal{F}$ restricted to $\mathcal{O}^{el}$. (See lemma 2.9 of \cite{chu2019homotopy}).

\end{itemize}

We can immediately see this as a special case of the notion of theory of the present paper as follows: Consider the functor $\mathcal{O}^{in} \to \Prsh \mathcal{O}^{el}$ that is obtained by composing the Yoneda embedding with the restriction functor:

\[ \mathcal{O}^{in} \to \Prsh \mathcal{O}^{in} \to \Prsh \mathcal{O}^{el} \]

The induced nerve functor $\Prsh \mathcal{O}^{el} \to \Prsh \mathcal{O}^{in}$ is equivalent to the fully faithful inclusion of the full subcategory of objects of $\Prsh \mathcal{O}^{in}$ that satisfies the Segal condition mentioned above. By definition the $\infty$-category of Segal $\mathcal{O}$-objects, we have a pullback:

\begin{equation}\label{diag:Seg_as_pb}
\begin{tikzcd}
 Seg_\mathcal{O} \ar[d] \ar[r] \ar[dr,phantom,"\lrcorner"very near start] & \Prsh \mathcal{O} \ar[d] \\
 \Prsh \mathcal{O}^{el} \ar[r] & \Prsh \mathcal{O}^{in}
\end{tikzcd}
\end{equation}

That is, $Seg_\mathcal{O}$ is the category of $\mathcal{O}$-models where $\mathcal{O}$ is seen as $\mathcal{O}^{in}$-theory for the canonical inclusion $\mathcal{O}^{in} \to \mathcal{O}$, and the dense functor $\mathcal{O}^{in} \to \Prsh \mathcal{A}$.

The condition that the categorical pattern $\mathcal{O}$ is \emph{extendable} (see Definition 8.5 of \cite{chu2019homotopy}) is equivalent, by Proposition 8.8 of \cite{chu2019homotopy} to the fact that the pullback diagram (\ref{diag:Seg_as_pb}) satisfies a Beck-Chevalley condition. That is, that the corresponding $\mathcal{O}^{in}$-Nervous monad on $\Prsh \mathcal{O}^{el}$ is a monad with arities in the sense of \cref{def:monad_with_arities}.

In particular, the mains results of Chu and Haugseng can be summarized in our language as:

\begin{itemize}
\item For an extendable algebraic pattern, the associated monad under the monad-theories correspondance is a parametric right adjoint cartesian monad on $\Prsh \mathcal{O}^{el}$,

\item conversely any such parametric right adjoint cartesian monad on a presheaf category can be obtained this way.

\end{itemize}

They also formulate a more precise form of this in terms of an equivalence of $\infty$-categories between parametric right adjoint cartesian monads and a certain subclass of algebraic pattern as Theorem 15.8.

In particular, all the examples of categorical patterns given in section 3 of \cite{chu2019homotopy} are examples of theories, or equivalently of nervous monads. This includes:

\begin{enumerate}

\item The free $\Gamma$-space monad (or equivalently $E_{\infty}$-space monad) on the $\infty$-category of spaces, which is described as the theory $\mathcal{A} \to \Gamma$ for $\mathcal{A}$ the category of finite sets and injections and $\Gamma$ the Segal category (i.e. the opposite category of pointed finite sets). See example 3.1 of \cite{chu2019homotopy}.

\item A ``free $n$-uple Segal spaces'' monad on the $\infty$-category $\Prsh (\Delta_{\leqslant 1})^n$ (Example 3.4 of \cite{chu2019homotopy}).

\item A ``free Rezk $\Theta_n$-space'' monad on the $\infty$-category of globular spaces, which are a model of $(\infty,n)$-categories. (Example 3.5 of \cite{chu2019homotopy}).

\item The category of dendroidal space is also obtained as the category of algebras for a monads on the category of presheaves on the category of corollas, see example 3.7 of \cite{chu2019homotopy}. Other kind of operads (cyclic, modular, properads, etc...) have similar description in the subsequent examples (example 3.8 to 3.11).
\end{enumerate}

\bibliography{biblio.bib}

\end{document}